\documentclass{article}
\usepackage[hmargin=1.2in,vmargin=1in]{geometry}

\usepackage{hyperref}
\newcommand*{\email}[1]{\href{mailto:#1}{\nolinkurl{#1}} } 
\usepackage[ruled,vlined]{algorithm2e}
\usepackage{biblatex}
\usepackage{lipsum}
\usepackage{amsfonts}
\usepackage{amssymb}
\usepackage{mathrsfs}
\usepackage{mathtools}
\usepackage{graphicx}
\usepackage{epstopdf}
\usepackage{algorithmic}
\usepackage{setspace}
\usepackage[dvipsnames]{xcolor}
\usepackage{enumitem}
\usepackage{accents}
\usepackage{enumitem}
\usepackage{amssymb}
\usepackage{microtype}
\usepackage{subcaption}
\usepackage{tikz}
\usepackage{soul}
\ifpdf
  \DeclareGraphicsExtensions{.eps,.pdf,.png,.jpg}
\else
  \DeclareGraphicsExtensions{.eps}
\fi

\ifcsname newsiamremark\endcsname%
    \newcommand{\titletext}{Random Gradient-Free Optimization in Infinite Dimensional Spaces}
    \newcommand{\authorsabbreviated}{C. Peixoto, D. Csillag, B. F. P. da Costa, Y. F. Saporito}
    \newcommand{\EMApname}{School of Applied Mathematics, Getulio Vargas Foundation}

    \newsiamremark{remark}{Remark}
    \newsiamremark{hypothesis}{Hypothesis}
    \newsiamremark{assumption}{Assumption}
    \newsiamremark{example}{Example}
    \crefname{hypothesis}{Hypothesis}{Hypotheses}
    \newsiamthm{claim}{Claim}
    \newsiamremark{fact}{Fact}
    \crefname{fact}{Fact}{Facts}

    \headers{Rand. Grad.-Free Optim. in Inf. Dim. Spaces}{\authorsabbreviated}
\else%
    \newcommand{\titletext}{Random Gradient-Free Optimization \\ in Infinite Dimensional Spaces}
    \newcommand{\authorsabbreviated}{C. Peixoto, D. Csillag, B. F. P. da Costa, Y. F. Saporito}
    \newcommand{\EMApname}{School of Applied Mathematics, Getulio Vargas Foundation}

    \usepackage{amsthm}
    \newtheorem{remark}{Remark}

    \newtheorem{theorem}{Theorem}
    \newtheorem{proposition}{Proposition}
    \newtheorem{lemma}{Lemma}
    \newtheorem{corollary}{Corollary}
\fi%

\title{\titletext\thanks{%
CP was partially supported by CAPES through the CAPES/PICME master's scholarship and FGV EMAp during his work in this manuscript. YFS was supported by FAPERJ (Brasil) through the Jovem Cientista do Nosso Estado Program (E-26/204.207/2025 (272760)) and by CNPq (Brasil) through the Productivity in Research Scholarship (305159/2025-9).
}}

\author{
Caio Peixoto\thanks{\EMApname, Rio de Janeiro, Brazil,  \email{caio.peixoto@fgv.br}, \email{daniel.csillag@fgv.br}, \email{bernardo.paulo@fgv.br}, \email{yuri.saporito@fgv.br}.}
\and Daniel Csillag$^\dagger$\!\!\!%
\and Bernardo F. P. da Costa$^\dagger$\!\!\!%
\and Yuri F. Saporito$^\dagger$
}

\usepackage{amsopn}
\DeclareMathOperator{\diag}{diag}

\makeatletter
\def\cref@getref#1#2{%
  \expandafter\let\expandafter#2\csname r@#1@cref\endcsname%
  \expandafter\expandafter\expandafter\def%
    \expandafter\expandafter\expandafter#2%
    \expandafter\expandafter\expandafter{%
      \expandafter\@firstoffive#2}}%
\def\cpageref@getref#1#2{%
  \expandafter\let\expandafter#2\csname r@#1@cref\endcsname%
  \expandafter\expandafter\expandafter\def%
    \expandafter\expandafter\expandafter#2%
    \expandafter\expandafter\expandafter{%
      \expandafter\@secondoffive#2}}%

\AtBeginDocument{%
   \def\label@noarg#1{%
    \cref@old@label{#1}%
    \@bsphack%
    \edef\@tempa{{page}{\the\c@page}}%
    \setcounter{page}{1}%
    \edef\@tempb{\thepage}%
    \expandafter\setcounter\@tempa%
    \cref@constructprefix{page}{\cref@result}%
    \protected@write\@auxout{}%
      {\string\newlabel{#1@cref}{{\cref@currentlabel}%
      {[\@tempb][\arabic{page}][\cref@result]\thepage}{}{}{}}}%
    \@esphack}%
  \def\label@optarg[#1]#2{%
    \cref@old@label{#2}%
    \@bsphack%
    \edef\@tempa{{page}{\the\c@page}}%
    \setcounter{page}{1}%
    \edef\@tempb{\thepage}%
    \expandafter\setcounter\@tempa%
    \cref@constructprefix{page}{\cref@result}%
    \protected@edef\cref@currentlabel{%
      \expandafter\cref@override@label@type%
        \cref@currentlabel\@nil{#1}}%
    \protected@write\@auxout{}%
      {\string\newlabel{#2@cref}{{\cref@currentlabel}%
      {[\@tempb][\arabic{page}][\cref@result]\thepage}{}{}{}}}%
    \@esphack}%
}

\newcommand{\defeq}{\triangleq}
\newcommand{\prob}{\mathbb{P}}
\newcommand{\ind}{\mathbf{1}}
\newcommand{\mean}{\mathbb{E}}

\newcommand{\trp}{\text{\tiny\sffamily T}}

\DeclarePairedDelimiter\abs{\lvert}{\rvert}%
\DeclarePairedDelimiter\norm{\lVert}{\rVert}%
\DeclarePairedDelimiter\dotprod{\langle}{\rangle}

\DeclareMathOperator{\closure}{cl}

\makeatletter
\let\oldabs\abs
\def\abs{\@ifstar{\oldabs}{\oldabs*}}
\let\oldnorm\norm
\def\norm{\@ifstar{\oldnorm}{\oldnorm*}}
\makeatother

\newcommand{\bzero}{\mathbf 0}
\newcommand{\ddrm}{\mathrm{d}}
\newcommand{\drm}{\gap \mathrm{d}}

\DeclareMathOperator*{\argmin}{arg\,min}

\DeclareMathOperator{\trace}{tr}
\DeclareMathOperator{\Trace}{Tr}

\newcommand{\Bcal}{\mathcal{B}}

\newcommand{\Ecal}{\mathcal{E}}
\newcommand{\Fcal}{\mathcal{F}}

\newcommand{\Hcal}{\mathcal{H}}

\newcommand{\Ncal}{\mathcal{N}}

\newcommand{\Scal}{\mathcal{S}}

\newcommand{\bbE}{\mathbb{E}}

\newcommand{\bbH}{\mathbb{H}}

\newcommand{\bbN}{\mathbb{N}}

\newcommand{\bbP}{\mathbb{P}}

\newcommand{\bbR}{\mathbb{R}}

\newcommand{\bA}{\mathbf{A}}
\newcommand{\bB}{\mathbf{B}}
\newcommand{\bC}{\mathbf{C}}
\newcommand{\bD}{\mathbf{D}}

\newcommand{\bI}{\mathbf{I}}

\newcommand{\bL}{\mathbf{L}}
\newcommand{\bM}{\mathbf{M}}

\newcommand{\bP}{\mathbf{P}}
\newcommand{\bQ}{\mathbf{Q}}
\newcommand{\bR}{\mathbf{R}}
\newcommand{\bS}{\mathbf{S}}
\newcommand{\bT}{\mathbf{T}}

\newcommand{\bc}{\boldsymbol{c}}

\newcommand{\be}{\mathbf{e}}

\newcommand{\bp}{\boldsymbol{p}}

\newcommand{\bw}{\boldsymbol{w}}
\newcommand{\bx}{\boldsymbol{x}}
\newcommand{\by}{\boldsymbol{y}}
\newcommand{\bz}{\boldsymbol{z}}

\newcommand{\risk}{\mathcal{R}}

\newcommand{\hstar}{h^{\star}}

\newcommand{\ustar}{u^{\star}}

\renewcommand{\hat}{\widehat}

\newcommand{\ghat}{\hat{g}}

\newcommand{\gbar}{\bar{g}}

\newcommand{\Htilde}{\tilde{H}}
\newcommand{\Omegatilde}{\tilde{\Omega}}
\newcommand{\hhat}{\hat{h}}
\newcommand{\Cminushalf}{C^{-\frac{1}{2}}}
\newcommand{\Chalf}{C^{\frac{1}{2}}}
\newcommand{\Cinv}{C^{-1}}

\newcommand{\xbar}{\bar{x}}
\newcommand{\Xbar}{\bar{X}}

\newcommand{\basis}{\mathcal{E}}
\newcommand{\prebasis}{\mathcal{B}}
\newcommand{\Kdist}{\bbP_{K}}
\newcommand{\vdist}{\bbP_v}
\newcommand{\vdistgivenK}{\bbP_{v\mid K}}
\newcommand{\vdistgivenk}{\bbP_{v\mid k}}

\newcommand{\gap}{\hspace{1.3pt}}

\newcommand{\kernel}{\mathrm{Matern}}

\DeclareMathOperator{\spann}{span}
\DeclareMathOperator{\Cov}{Cov}

\DeclareMathOperator{\EnforceBC}{EnforceBC}
\DeclareMathOperator{\PoissonPofTwo}{PoissonP2}
\DeclareMathOperator{\Poisson}{Poisson}
\newcommand{\Lip}{\mathrm{Lip}}

\newcommand{\superdomain}{\Xi}

\renewcommand{\algorithmiccomment}[1]{\hfill\texttt{/\kern-.25em/} \kern-.1em #1}

\newcommand{\algoname}{Functional Gradient-Free Descent}
\newcommand{\algoacronym}{FGFD}

\date{December 2025}

\usepackage{cleveref}

\definecolor{blue1}{HTML}{42d4f4}
\definecolor{blue2}{HTML}{4363d8}
\definecolor{blue3}{HTML}{000075}
\definecolor{magma1}{rgb}{0.232077, 0.059889, 0.437695}
\definecolor{magma2}{rgb}{0.716387, 0.214982, 0.47529}
\definecolor{magma3}{rgb}{0.994738, 0.62435, 0.427397}

\allowdisplaybreaks

\begin{document}

\maketitle

\begin{abstract}
    We propose a new gradient-free method for infinite-dimensional optimization in Hilbert spaces that requires only the computation of directional derivatives. Though functional optimization is often solved through finite-dimensional gradient descent over a parametrization, such as neural networks, we instead propose to leverage the functional nature of the optimization problem to enable provable guarantees. However, infinite-dimensional gradients are often hard to compute in practice, rendering naïve functional gradient descent intractable. To overcome this limitation, our framework leverages only directional derivatives and a pre-basis for the Hilbert space, i.e., a linearly independent set whose span is dense. This resolves the tractability issue, as pre-bases are much more accessible than full orthonormal bases or reproducing kernels --- which may not even exist --- and individual directional derivatives can be computed using automatic differentiation. We showcase the use of our method to solve partial differential equations \emph{à la} physics-informed neural networks (PINNs), where it effectively enables provable convergence.
\end{abstract}

\section{Introduction}
\label{sec:introduction}

The solution of functional optimization problems is core to many applications, with prominent examples including machine learning tasks such as supervised learning and generative modeling, optimization-based PDE solvers and broader inverse problems. Formally speaking, these are of the form
\begin{equation}
    \label{eq:main problem}
    \argmin_{h \in \Hcal} \risk(h),
\end{equation}
where $\Hcal$ is a separable Hilbert space (e.g., $L^2$ space, or a Sobolev space), and $\risk : \Hcal \to \bbR$ is a Fréchet differentiable risk functional, often convex.

To solve problems like \eqref{eq:main problem}, practitioners typically introduce a finite-dimensional parametrization $\theta \mapsto h_\theta$ of the functions $h \in \Hcal$ and then optimize over $\theta \in \bbR^D$ using some variant of gradient descent.
This approach, however, is limited by the representation capacity of the chosen parametrization; moreover, because the parametrization is typically nonlinear (e.g., neural networks), even if $\risk$ is convex, the composition $\theta \mapsto \risk(h_\theta)$ is most likely not, increasing the complexity of the resulting optimization problem.
To overcome these issues, we perform gradient-based optimization \emph{directly in the function space} $\Hcal$, since its Hilbert space structure gives the existence of gradients as Riesz representers of Fréchet derivatives.
By doing so, we may retain convexity of the risk and also avoid the limited representation power of any specific parametrization.

However, gradient-based optimization in infinite-dimensional Hilbert spaces is challenging.
For one, functional gradients themselves are often hard to compute exactly, being defined as solutions to variational problems involving the directional derivatives of the risk.
Secondly, it is not generally clear how to represent infinite dimensional objects computationally.
Thirdly, in the realm of stochastic gradient methods, controlling the variance of an infinite dimensional random element usually requires more care than that of a finite dimensional random vector.

In this paper, we adopt the functional optimization perspective, but we bypass exact gradient computation by taking inspiration from random gradient-free optimization methods.
This is a class of algorithms that build stochastic gradient estimators based only on directional derivative computations, which are readily available through automatic differentiation.
We adapt this framework to be applicable in infinite dimensions, yielding \algoname\ (\algoacronym), a general algorithm for minimizing Fréchet differentiable risk functionals over separable Hilbert spaces, with convergence guarantees in the convex case.
\looseness=-1

\subsection{A motivating example: solving PDEs through gradient descent}
\label{sec:motivation}

Consider solving the following PDE with Dirichlet boundary conditions:
\begin{equation}
    \label{eq:motivating pde}
    \begin{cases}
        L[u](\bx) = 0 \quad &\bx \in \Omega, \\
        u(\bx) = f(\bx) \quad &\bx \in \partial \Omega,
    \end{cases}
\end{equation}
where $\Omega \subset \bbR^d$ is a bounded, open set with a smooth boundary and $L$ is an $\ell$-th order differential operator.
For simplicity, in this section we will assume that $L$ is a linear operator with bounded coefficients.
Recently, a fast-growing body of research \cite{cuomo2022, raissi2024pinn, pinn, dgm2018} has been dedicated to study deep learning algorithms that find approximate solutions to \eqref{eq:motivating pde}.
These algorithms work by minimizing a functional of the form\looseness=-1
\begin{equation}
    \label{eq:motivating example risk}
    \risk(h) = \frac{1}{2} \norm{L[h]}_{L^2(\Omega)}^2 + \frac{1}{2} \norm*{h - f}_{L^2(\partial \Omega)}^2,
\end{equation}
where $h = h_\theta$ is taken to be a neural network with fixed architecture, parameterized by $\theta \in \bbR^D$ for some large $D \in \bbN$.
The idea is to minimize $\risk(h_\theta)$ over $\theta$ through variants of gradient descent in the parameter space.

We propose to see \eqref{eq:motivating example risk} as acting directly on functions, as opposed to a loss over finite-dimensional parameters.
Our main motivation for doing so is preserving the convexity of the risk when $L$ is linear.\footnote{If $L$ is nonlinear we may lose convexity, but still obtain a problem that is simpler than what a non-linear parametrization would result in.} %
This is also aligned with a recent surge in the \emph{first optimize, then discretize} paradigm in the scientific ML literature, cf. \cite{muller23, position-sciml}, which tries to emulate functional optimization algorithms when optimizing a parametric ansatz $h_\theta$ to solve \eqref{eq:motivating pde}.
The functional perspective has also influenced recent works in statistical inverse problems \cite{fonseca2022statistical} and nonparametric instrumental variable regression \cite{fonseca2024nonparametric}.

To proceed, we must first specify the domain $\Hcal$ of the risk $\risk$.
Due to the
nature of problem \eqref{eq:motivating pde}, we take $\Hcal = \bbH^\ell(\Omega)$,
the Sobolev space of order $\ell$ and exponent $2$ in $\Omega$ --- that is, the space of all $h \in L^2(\Omega)$ whose weak derivatives exist up to order $\ell$ and lie in $L^2(\Omega)$, which we endow with the usual inner product $\dotprod{h, h'}_{\bbH^\ell(\Omega)}~=~\sum_{j} \dotprod{D^{(j)}h, D^{(j)}h'}_{L^2(\Omega)}$.\footnote{We denote the $j$-th derivative of $h$ by $D^{(j)}h$, which we interpret as a $d^j$-dimensional array of real numbers.}
Thus, solving the PDE \eqref{eq:motivating pde} becomes the optimization problem
\begin{equation}
    \label{eq:motivating problem}
    \argmin_{h \in \bbH^\ell(\Omega)} \risk(h),
\end{equation}
where $\risk$ is given by \eqref{eq:motivating example risk}.

To solve this problem under the functional paradigm, the most direct approach would be to
employ gradient descent in $\bbH^\ell(\Omega)$.
By definition, if $\risk$ is Fréchet differentiable, the gradient $\nabla \risk(h)$ is the unique element of $\bbH^\ell(\Omega)$ satisfying $\dotprod{\nabla \risk(h), v}_{\bbH^\ell(\Omega)} = D \risk (h ; v)$ for all $v \in \bbH^\ell(\Omega)$, where
\begin{equation}
    \label{eq:direc deriv def}
    D \risk (h ; v) = \lim_{\delta \to 0} \frac{\risk(h + \delta v) - \risk(h)}{\delta}
\end{equation}
is the directional derivative of $\risk$ at $h$ in the direction $v$.
Thus, to obtain $\nabla \risk (h)$ we first need to compute $D \risk (h; v)$, which is straightforward in this case due to the linearity of $L$:
\begin{equation}
    \label{eq:direc deriv linear case}
    D \risk (h ; v)
    = \int_\Omega L[h](\bx) L[v](\bx) \drm \bx
    + \int_{\partial \Omega} (h(\bx) - f(\bx)) v(\bx) \gap \sigma(\ddrm \bx),
\end{equation}
where $\sigma$ is the surface measure on $\partial \Omega$.

We can now see that the functional $D\risk(h; \cdot)$ is linear and continuous in the $\bbH^\ell(\Omega)$-norm.
Furthermore, the map $h \mapsto D \risk (h; \cdot)$, from $\bbH^\ell(\Omega)$ to its dual, is also continuous, implying that $\risk$ is Fréchet differentiable \cite{pathak2018}.
Therefore, the gradient $\nabla \risk (h)$ is well-defined, and finding it is equivalent to solving the following problem:
\begin{equation*}
    \begin{aligned}
        &\hspace{2cm}\text{Find $g \in \bbH^\ell(\Omega)$ such that}\\
        &
        \sum_{j = 0}^\ell \int_{\Omega} D^{(j)} g(\bx) \cdot D^{(j)} v(\bx) \drm \bx
        = \int_\Omega L[h](\bx) L[v](\bx) \drm \bx
        + \!\!\int_{\partial \Omega} (h(\bx) - f(\bx)) v(\bx) \gap \sigma(\ddrm \bx), \\
        &\hspace{7cm}\text{for all $v \in \bbH^\ell(\Omega)$}.
    \end{aligned}
\end{equation*}
This is a variational problem for which no general closed form solution exists, creating the need to look beyond exact gradients in order to solve \eqref{eq:motivating problem}.

\subsection{Background: random gradient-free methods}

To avoid the computation of exact gradients, an existing branch of the literature constructs ``random gradient-free methods,'' which leverage stochastic gradient estimators based on computations of directional derivatives.
Hence, though these methods are not derivative-free in a strict sense, they are gradient-free.
So far, this approach has only been considered for \emph{finite dimensional} optimization problems. These methods exhibit prohibitive limitations when generalized to infinite-dimensional spaces.

If $\Hcal = \bbR^d$ for some $d \in \bbN$, then random gradient-free methods estimate the gradient $g \defeq \nabla \risk(h)$ as
\begin{equation}
    \label{eq:ghat definition}
    \ghat = \frac{1}{M} \sum_{m=1}^{M} D \risk (h ; v_m) v_m,
\end{equation}
where $v_1,\ldots,v_M$ are i.i.d.\ vectors in $\bbR^d$ with zero mean and identity covariance matrix, cf. \cite{nesterov2017}.
This is thus an unbiased estimator of $g$ that, under additional assumptions on $v$, has finite second moment: $\bbE[\norm{\ghat}^2] < +\infty$. For example, in the convenient case where $v_i \sim \Ncal(0,\bI_d)$, one may show that $\mean[\norm{\ghat}^2] = \norm*{g}^2 \left( \frac{M + d + 1}{M}\right)$.
This setup already exhibits a couple of issues that may arise when $\Hcal$ is high/infinite dimensional:

\begin{enumerate}[label=(\roman*)]

\item The expression for $\mean[\norm*{\ghat}^2]$ presented above suggests taking the sample size $M$ to be proportional to $d$. As $d \to +\infty$, this becomes infeasible. \label{en:infinite sample size}

\item Many of the usual choices for the distribution of the $v_i$'s have no infinite-dimensional analogue: there is no isotropic Gaussian distribution \cite{daprato2006}, no uniform measure over the sphere \cite{yamasaki1985}, and no Rademacher distribution (since sequences of $\pm 1$'s are not square-summable).
These are distributions used, for example, in \cite{kozak2021, nesterov2017, stich2013} for the finite dimensional scenario.
\label{en:common distributions fail}

\item More broadly, if a random element $v$ in an infinite dimensional Hilbert space satisfies $\bbE[\norm{v}^2] < +\infty$, then it necessarily has a compact covariance operator, which thus cannot be the identity --- a necessary condition for the unbiasedness of $\ghat$ in \eqref{eq:ghat definition}.
In fact, since eigenvalues of compact operators concentrate around zero \cite[Theorem~4.25]{rudin}, we have $\norm{T - I} \geq 1$ for any compact operator $T: \Hcal \to \Hcal$.\looseness=-1 \label{en:compact covariance operator}

\item Generally speaking, it is difficult to design unbiased gradient estimators $\hat{g}$ whose second moment is guaranteed to be finite for any gradient $g$. \label{en:finite variance all g}

\end{enumerate}

In this paper, we resolve these issues by proposing a novel random gradient-free algorithm that is suitable for infinite-dimensional optimization.
By first sampling a finite dimension $K$ and then sampling random directions $v_i$ from a specific $K$-dimensional subspace of $\Hcal$, we are able to overcome issues \ref{en:common distributions fail} and \ref{en:compact covariance operator};
then, by choosing the sample size $M$ proportional to the random dimension, we bypass issue \ref{en:infinite sample size}.
Finally, by introducing an appropriate preconditioning of the gradient descent process we are able to control the variance of our gradient estimates, thus resolving issue \ref{en:finite variance all g}.

\paragraph{Notation} We denote the set of positive integers by $\bbN$. Given a real, separable Hilbert space $(\Hcal, \dotprod{\cdot, \cdot})$ and a finite dimension $d$, the elements of $\Hcal$ are denoted by regular font Latin letters $f, g, h, \ldots$, whereas the elements of $\bbR^d$ are denoted by boldface Latin letters $\bx, \by, \bz, \ldots$
Regular font Greek letters will always denote scalars.
For $h, h' \in \Hcal$, we write $h \otimes h'$ to denote the rank-one map $(h \otimes h')v = \dotprod{h', v} h$.
If $\basis = \{e_1, e_2, \ldots \}$ is an orthonormal basis for $\Hcal$, $\{\gamma_i\}_{i \in \bbN}$ is a positive, bounded sequence and\footnote{Here and in similar situations where there is a series of operators, we define the series as the limit of the partial sums in the \emph{strong operator topology}, not the uniform operator topology.}~$C = \sum_{i=1}^\infty \gamma_i (e_i \otimes e_i) $ is a diagonal operator, the powers $C^s$, for $s \in \bbR$, are defined through the usual functional calculus: $C^s = \sum_{i=1}^\infty (\gamma_i)^s (e_i \otimes e_i)$.
These are possibly unbounded operators with restricted domains.

\section{\algoname}

We assume to have access to a \emph{pre-basis} $\prebasis = \{ b_1, b_2, \ldots \} \subset \Hcal$, i.e., a linearly independent subset of $\Hcal$ whose span is dense in $\Hcal$.
With $\prebasis$ in hand, we define the finite-dimensional subspaces $\prebasis_k = \spann\{b_1, \ldots, b_k\}$, for $k \in \bbN$.
As a theoretical tool, it will also be useful to refer to the orthonormalization $\basis = \{e_1, e_2, \ldots\}$ of the pre-basis, defined through Gram-Schmidt.
Since $\prebasis$ has dense span, $\basis$ is an actual orthonormal basis of $\Hcal$. We emphasize that we do not assume an orthonormal basis to be available in closed form, rendering our approach applicable to spaces beyond the simple cases where such a basis is known.\footnote{For example, in our motivating application of optimization over Sobolev spaces, there are no explicitly-known orthonormal bases for general domains $\Omega$ and differentiability orders $\ell$.}\looseness=-1

Fix a probability space $(\Scal, \Fcal, \prob)$.
At the $n$-th optimization step, our procedure starts by 
sampling a random finite dimension $K_n$ from some distribution $\Kdist$ in $\bbN$ with infinite support.
Given $K_n=k$, it then performs the updates
\begin{equation}\label{eq:our-update}
    h_{n+1} = h_n - \alpha_n \gap \ghat_n(k)
    \qquad \text{with} \qquad
    \ghat_n(k) = \frac{\lambda_k}{M_k} \sum_{m=1}^{M_k} D \risk (h_n ; v_m) v_m,
\end{equation}
where $M_k \in \bbN$ denotes the sample-size for the stochastic gradient approximation, $v_1,\ldots,v_{M_k}$ are i.i.d.\ random directions in $\Hcal$ sampled from a distribution $\vdistgivenk$ supported in $\prebasis_k$, and $\lambda_k \in \bbR_+$ corresponds to a preconditioner on the gradient.
We require $\{ \lambda_i \}_{i\in\bbN}$ to be positive and bounded.

We seek to design a distribution for the $v_i$'s such that $\ghat_n(K_n)$ is an unbiased preconditioned gradient, i.e., letting $g_n \defeq \nabla \risk(h_n)$, we want
the expectation of $\ghat_n(K_n)$ over the joint distribution of $K_n$ and the $v_i$'s, conditional on $h_1, \ldots, h_n$, to be $C g_n$ for a positive-definite operator $C : \Hcal \to \Hcal$, while ensuring finite second moment in an appropriate sense.
This is straightforward in finite dimension, but challenging in infinite dimensional spaces.
To this end, we construct the conditional distributions $\vdistgivenk$ so that their marginal $\vdist$ over $K_n$ has identity covariance; this guarantees preconditioned unbiasedness, and with additional conditions (to be derived in \Cref{thm:ghat variance} and corollaries) will also ensure finite second moment.

To attain this, it will be useful to refer to the QR decomposition of quasi-matrices.
Let $B_k$ be the quasi-matrix $[b_1, \ldots, b_k]$, i.e., the operator $B_k: \bbR^k \to \Hcal$ given by
\begin{equation*}
    B_k = \sum_{i=1}^k b_i \otimes \be_i,
\end{equation*}
where $\{\be_1, \ldots, \be_k \}$ is the canonical basis of $\bbR^k$.
The QR decomposition can be naturally extended to quasi-matrices, resulting in $B_k = Q_k \bR_k$, where $Q_k = [e_1, \ldots, e_k]$ is an orthogonal quasi-matrix
and $\bR_k \in \bbR^{k \times k}$ is an upper-triangular matrix \cite{trefethen2010}.

Let $K \sim \Kdist$ and define the tail probabilities $t_i = \prob [K \geq i]$, as well as the matrix $\bT_k = \diag (t_1, \ldots, t_k)$.
Since $\Kdist$ has infinite support, we have $t_i > 0$ for all $i \in \bbN$.
We then take $\vdistgivenk$ to be the distribution of
\begin{equation}
    \label{eq:v definition}
    v = B^{\vphantom{-\frac{1}{2}}}_k \bR^{-1}_k \bT_k^{-\frac{1}{2}} \bz,
    \quad
    \text{where}
    \quad
    \bz \sim \Ncal(\bzero, \bI_k).
\end{equation}

With this choice, it follows:

\begin{proposition}
    \label{prop: unbiased precond gradient}
    The distributions $\vdistgivenk$ and $\vdist$ both have zero mean and satisfy
    \vspace{-2mm}
    \begin{equation}
        \label{eq:covariances}
        \Cov[\vdistgivenk] = Q_k^{\vphantom{-1}} \bT_k^{-1} Q_k^{\vphantom{-1}*},
        \qquad\quad
        \Cov[\vdist] = I.
    \end{equation}
    Consequently, $ \mean_n[\gap \ghat_n(K_n)] = C \gap \nabla \risk (h_n)$,  where $\mean_n$ denotes the conditional expectation given $h_1, \ldots, h_n$, and $C : \Hcal \to \Hcal$ is the diagonal operator
    \begin{equation}
        \label{eq:C definition}
        C = \sum_{i=1}^\infty \gamma_i \gap (e_i \otimes e_i),
    \end{equation}
    with $\gamma_i \defeq \mean [\lambda_K \mid K \geq i]$. 
\end{proposition}
\begin{proof}
    See \Cref{sec:proof of unbiased precond gradient}.
\end{proof}

\begin{remark}
    Distributions of $v$ for which $\mean[\norm*{v}^2] < +\infty$ are said to be of \emph{strong second order}, and must have a compact covariance operator, cf. issue \ref{en:compact covariance operator} in \Cref{sec:introduction}.
    Since $\Hcal$ is infinite dimensional, the identity operator  is not compact, implying that our proposed distribution $\vdist$ is not of strong second order.
    It is, however, of \emph{weak second order}, meaning that it satisfies $\mean[\dotprod{v, h}^2] < +\infty$ for any $h \in \Hcal$.
    Such distributions still have covariance operators, which do not need to be compact, cf. \cite{baker1981, vakhania1978}.
\end{remark}

Leveraging this, we obtain \Cref{algo:our algorithm}, which solves \eqref{eq:main problem} by employing update \eqref{eq:our-update} with the distribution $\vdistgivenk$ specified in \eqref{eq:v definition}.
The analysis of our algorithm is inspired by that of preconditioned gradient and variable metric methods, see~\hbox{\cite[Chapter~4]{duchi2018}} and the references therein.
In this direction, a central object for the upcoming results is the (diagonal) inverse square-root operator
$\Cminushalf : C^{\frac{1}{2}} (\Hcal) \to \Hcal$.
This operator is well-defined because both $C$ and $\Chalf$ are positive definite operators: we require $\lambda_i > 0$ for all $i$, implying that the diagonal of $C$ is strictly positive.
With $\Cminushalf$ we can define the inner product $\dotprod{h, h'}_C \defeq \dotprod{\Cminushalf h, \Cminushalf h'}$, for $h, h' \in \Chalf (\Hcal)$, as well as the norm $\norm{h}_C^2 = \dotprod{h, h}_C$, which controls the behavior of our convergence bounds.
It is straightforward to check that $\Chalf (\Hcal)$ contains the orthonormal basis $\basis$, making $\Cminushalf$ a densely-defined linear operator in $\Hcal$.

\begin{remark}
    \label{eq:cameron martin}
    We note that $C$ is the covariance operator of $\sqrt{\lambda_K} v$, for $v$ given by \eqref{eq:v definition}.
    Though the distribution of $\sqrt{\lambda_K} v$ is not Gaussian, in the context of Gaussian measures the subspace  $\Chalf(\Hcal)$ is called the Cameron-Martin space of the Gaussian with zero mean and covariance $C$, cf. \cite{vakhania1978}.
\end{remark}
\begin{algorithm}[t]
    \caption{\algoname \ (\algoacronym)}
    \label{algo:our algorithm}
    \hspace*{-0.4cm}
    \begin{minipage}{.98\textwidth}
    \setstretch{1.2}
    \begin{algorithmic}
        \vspace{.05em}
        \STATE \textbf{Parameters:}
        Distribution $\Kdist$ over $\bbN$, sequence of sample sizes $\{ M_k \}_{k\in\bbN} \subset \bbN$, sequence of preconditioning parameters $\{\lambda_k\}_{k\in\bbN} \subset \bbR_+$, sequence of learning rates $\{\alpha_{n}\}_{n\in\bbN} \subset \bbR_+$, number of iterations $N \in \bbN$, starting point $h_1 \in \Chalf(\Hcal)$.
        \STATE \textbf{Input:} Directional derivative $D \risk : \Hcal \times \Hcal \to \bbR$, pre-basis $\prebasis = \{ b_1, b_2, \ldots \}$, inner product $\dotprod{\cdot, \cdot}$ (for computing QR decompositions).
        \STATE \textbf{Output:} $h_1, \ldots, h_{N+1} \in \Hcal$.
        \vspace{.5em}
        \FOR{$1 \leq n \leq N$}
            \STATE $k \gets $ sample from $\Kdist$
            \vspace{.5em}
            \STATE $B_k \gets [b_1, \ldots, b_k]$
            \STATE $\bT_k \gets \diag (t_1, \ldots, t_k)$
            \STATE $\bR_k \gets$ R matrix from QR decomposition of $B_k$ \COMMENT{Can be precomputed}
            \vspace{.5em}
            \STATE $\bz_1, \ldots, \bz_{M_k} \gets$ i.i.d.\ sample from $\Ncal(\bzero, \bI_k)$
            \STATE $v_{m} \gets B^{\vphantom{-\frac{1}{2}}}_k \bR^{-1}_k \bT_k^{-\frac{1}{2}} \bz_m, \ 1 \leq m \leq M_k$
            \STATE $\ghat_n \gets \frac{\lambda_k}{M_k} \sum_{m=1}^{M_k} D \risk (h_{n}; v_{m}) v_{m}$
            \vspace{.5em}
            \STATE $h_{n+1} \gets h_{n} - \alpha_n \ghat_n$
        \ENDFOR
        \vspace{.15em}
    \end{algorithmic}
    \end{minipage}
\end{algorithm}

We can now state our first convergence bound.
To simplify the notation we will write $\ghat_n$ instead of $\ghat_n (K_n)$, leaving the dependence of $\ghat_n$ on $K_n$ implicit.

\begin{theorem}
    \label{thm:starting bound}
    Assume that $\risk$ is convex, and
    let $h_1, \ldots, h_N$ be the output sequence of \Cref{algo:our algorithm}.
    Define the averaged estimate $ \hhat = (\sum_{n=1}^N \alpha_n h_n) / \sum_{n=1}^N \alpha_n$.
    Then, if $\mean[\norm*{\ghat_n}_C^2] < +\infty$ for all $n$, taking any $h \in C^{\frac{1}{2}}(\Hcal)$, we have:
    \begin{equation}
        \label{eq:intermediate risk bound}
        \mean[\risk(\hhat)] - \risk(h)
        \leq
        \frac{
            \frac{1}{2}\norm*{h - h_1}_{C}^2
        }{\sum_{n=1}^N \alpha_n}
        + \frac{
            \frac{1}{2}\sum_{n=1}^N \alpha_n^2 \mean[ \norm*{\ghat_n}_{C}^2 ]
        }{\sum_{n=1}^N \alpha_n}.
    \end{equation}
\end{theorem}
\begin{proof}
    See \Cref{sec:proof starting bound}.
\end{proof}

As is evident from \Cref{thm:starting bound}, it is essential that the second moment term $\mean[ \norm*{\ghat_n}_{C}^2 ]$ be finite. By the tower property, a necessary condition would then be $\mean_n[ \norm*{\ghat_n}_{C}^2 ] \in L^1(\Scal, \Fcal, \prob)$, where, we recall, $\mean_n$ denotes the conditional expectation given $h_1, \ldots, h_n$. 
If $\Hcal$ were finite-dimensional, then usual choices for $\vdist$ have finite moments of all orders, trivially ensuring a finite second moment for $\ghat_n$;
however, in infinite-dimensional spaces this is substantially harder, and is only possible due to (i) the particular design of our $\vdist$, and (ii) the introduction of the preconditioning $\lambda_k$, as we will now show.

\begin{theorem}
    \label{thm:ghat variance}
    Recall that $\gamma_i = \mean[\lambda_K \mid K \geq i]$ and let $g_n \defeq \nabla \risk (h_n)$.
    The second moment of $\ghat_n$ then satisfies
    \begin{equation}
    \label{eq:ghat variance}
            \mean_n\!\!\left[
                \norm*{\ghat_n}^2_{C}
            \right]
            \kern-.2em=\kern-.1em \mean_n\!\!\left[
                \kern-.1em
                \frac{\lambda_K^2}{M_K}
                \kern-.2em
                \left(
                    \sum_{i=1}^K \frac{1}{t_i \gamma_i}
                    \kern-.1em
                \right)
                \kern-.4em
                \left(
                    \sum_{i=1}^K \frac{\dotprod{g_n, e_i}^2}{t_i}
                    \kern-.1em
                \right)
                \kern-.2em
            \right]
            \kern-.3em+\kern-.1em \mean_n\!\!\left[
                \lambda_K^2
                \kern-.2em
                \left(
                    \kern-.2em
                    1 \kern-.2em+\kern-.2em \frac{1}{M_K}
                    \kern-.2em
                \right)
                \kern-.4em
                \left(
                    \sum_{i=1}^K \frac{\dotprod{g_n, e_i}^2}{t_i^2 \gamma_i}
                \right)
                \kern-.2em
            \right].
    \end{equation}
\end{theorem}
\begin{proof}
    See \Cref{sec:proof variance}.
\end{proof}

Let us start by considering what happens when there is no preconditioning:
\begin{corollary}
    \label{cor:no precond high var}
    Let $\lambda_k \equiv 1$.
    Then, independent of the choice of $M_k$, a necessary condition for $\mean_n [\norm*{\ghat_n}^2_C] < +\infty$ is
    \begin{equation}
        \label{eq:finite ellipsoid gradient norm}
        \sum_{i=1}^\infty \frac{\dotprod{g_n, e_i}^2}{t_i} < +\infty.
    \end{equation}
\end{corollary}
\begin{remark}
    Note that if $\lambda_k \equiv 1$, then $C = I$ and $\norm*{\cdot}_C = \norm*{\cdot}$ meaning that, without preconditioning, we cannot guarantee $\mean_n[\norm*{\ghat_n}^2] < +\infty$, for every $g_n \in \Hcal$.
\end{remark}
\begin{proof}
    From \eqref{eq:ghat variance}, it is immediate that
    \begin{equation*}
        \mean_n[\norm*{\ghat_n}^2_C]
        \geq 
        \mean_n\!\!\left[
            \lambda_K^2
            \left(
                1 + \frac{1}{M_K}
            \right)
            \left(
                \sum_{i=1}^K \frac{\dotprod{g_n, e_i}^2}{t_i^2 \gamma_i}
            \right)
        \right]
        \geq 
        \mean_n\!\!\left[
            \sum_{i=1}^K \frac{\dotprod{g_n, e_i}^2}{t_i^2}
        \right],
    \end{equation*}
    where we have substituted $\lambda_k \equiv 1$.
    Now compute
    \begin{equation*}
        \mean_n\!\!\left[
            \sum_{i=1}^K \frac{\dotprod{g_n, e_i}^2}{t_i^2}
        \right]
        = 
        \mean_n\!\!\left[
            \sum_{i=1}^\infty \ind\{ K \geq i\}\frac{\dotprod{g_n, e_i}^2}{t_i^2}
        \right]
        = 
        \sum_{i=1}^\infty \frac{\dotprod{g_n, e_i}^2}{t_i},
    \end{equation*}
    implying that $\sum_{i=1}^\infty \dotprod{g_n, e_i}^2 / t_i$ is a lower bound for $\mean_n[\norm*{\ghat_n}^2_C]$.
\end{proof}
Therefore, if $\lambda_k \equiv 1$, to obtain a useful bound in \Cref{thm:starting bound} we would need to assume \eqref{eq:finite ellipsoid gradient norm} for every gradient $g_n$ in \Cref{algo:our algorithm}, which is rather restrictive.
Fortunately, there is a better choice of $\lambda_k$ and $M_k$ which avoids this issue altogether.
First, a useful lemma:
\begin{lemma}
    \label{lem:t_i bound}
    For any distribution $\Kdist$ over $\bbN$ we have
    \begin{equation}
        \label{eq:t_i bound}
        \frac{t_i}{2} \leq \mean[t_K \mid K \geq i] \leq t_i,
        \quad
        \text{for all}
        \quad
        i \in \bbN.
    \end{equation}
\end{lemma}
\begin{proof}
    The upper bound is obvious, since $t_i = \prob[K \geq i]$ is decreasing in $i$.
    To obtain the lower bound, expand the expectation and factor out a telescoping sum:
    \begin{align*}
        \mean[t_K \mid K \geq i]
        = \sum_{j=i}^\infty t_j \frac{\prob[K = j]}{t_i}
        &= \frac{1}{t_i}\sum_{j=i}^\infty t_j (t_j - t_{j+1}) \\
        &= \frac{1}{t_i}\sum_{j=i}^\infty \frac{1}{2} \big(t_j^2 - t_{j+1}^2 \big) + \frac{1}{2} (t_j - t_{j+1})^2 \\
        &\geq \frac{1}{t_i}\sum_{j=i}^\infty \frac{1}{2} \big(t_j^2 - t_{j+1}^2 \big)
        = \frac{t_i^2}{2t_i} = \frac{t_i}{2},
    \end{align*}
    which follows from $t_i \to 0$ as $i \to +\infty$.
\end{proof}

As an immediate consequence, we have the following second moment bound:
\begin{corollary}
    \label{cor:variance bound}
    Choose $\lambda_k = t_k$ and $M_k = \lceil k / c \rceil$, for some $c > 0$.
    Then %
    \begin{equation}
        \label{eq:ghat variance bound}
        \mean_n [\norm*{\ghat_n}^2_C]
        \leq 2 (1 + 2 c) \norm{g_n}^2.
    \end{equation}
\end{corollary}
\begin{proof}
    From \eqref{eq:ghat variance}, we can see that $\mean_n[\norm*{\ghat_n}^2_C]$ is decreasing in $M_k$, so we can obtain  an upper bound by substituting $M_k = k / c$.
    Then, by \Cref{lem:t_i bound}, we have
    \begin{align*}
            \mean_n  \kern-.3em\left[
                \norm*{\ghat_n}^2_{C}
            \right]
            &\leq \mean_n \kern-.2em\left[
                \frac{c \gap t_K^2}{K}\kern-.2em
                \left(
                    \sum_{i=1}^K \frac{2}{ \gap t_i^2}
                \right)
                \!\!\left(
                    \sum_{i=1}^K \frac{\dotprod{g_n, e_i}^2}{t_i}
                \right)\kern-.1em
            \right]
            \!+ \mean_n \kern-.3em\left[
                t_K^2\kern-.2em
                \left(
                    1 + \frac{c}{K}
                \right)
                \!\!\left(
                    \sum_{i=1}^K \frac{2 \dotprod{g_n, e_i}^2}{t_i^3}
                \right)
            \right] \\
            &\leq 2 c \gap \mean_n\!\! \left[
                \sum_{i=1}^K \frac{\dotprod{g_n, e_i}^2}{t_i}
            \right]
            + 2 \left( 1 + c \right) \mean_n\!\! \left[
                \sum_{i=1}^K \frac{\dotprod{g_n, e_i}^2}{t_i}
            \right] \\
            &= 2 \left( 1 + 2 c \right) \norm{g_n}^2,
    \end{align*}
    where we have used the fact that $t_K \leq t_i$ if $i \leq K$.
\end{proof}

Thus, combining \Cref{thm:starting bound} and \Cref{cor:variance bound}, we obtain our main results -- one under general convexity, and one that holds under an additional assumption of smoothness (i.e., that the gradient $\nabla \risk$ is Lipschitz):

\begin{theorem}
    \label{thm:risk bound}
    Under the same setting as \Cref{thm:starting bound}, take $\lambda_k = t_k$ and $M_k = \lceil k/c \rceil$ for some $c > 0$. %
    Then, for any $h \in C^{\frac{1}{2}}(\Hcal)$,
    \begin{equation}
        \label{eq:risk bound}
        \mean [\risk(\hhat)] - \risk(h)
        \leq
        \frac{
            \frac{1}{2} \norm*{h - h_1}^2_C
        }{\sum_{n=1}^N \alpha_n}
        + (1 + 2c)
        \frac{
             \sum_{n=1}^N \alpha_n^2 \mean[\norm*{g_n}^2]
        }{\sum_{n=1}^N \alpha_n}.
    \end{equation}
\end{theorem}
\begin{proof}
    Direct consequence of \eqref{eq:intermediate risk bound} and \eqref{eq:ghat variance bound}.
\end{proof}

\begin{theorem}
\label{thm:risk bound smooth}
    Under the same setting as \Cref{thm:risk bound}, assume also that $\nabla \risk$ is Lipschitz, with Lipschitz constant $\Lip(\nabla \risk)$, and that there exists $\hstar \in \argmin_{h \in \Hcal} \risk(h)$ such that $\hstar \in \Chalf(\Hcal)$.
    If we choose learning rates satisfying $\alpha_n \leq 1/\bigl( (1 + 2c) \gap \Lip(\nabla \risk) \bigr)$ for all $1 \leq n \leq N$, then we obtain
    \begin{equation}
        \label{eq:risk bound smooth}
        \mean[\risk(\hhat)] - \risk(\hstar) \leq
        \norm*{\hstar - h_1}_C^2
        \left[ 
            \frac{1}{2 \sum_{n=1}^N \alpha_n} + \frac{\sum_{n=1}^N \alpha_n^2}{\sum_{n=1}^N \alpha_n} \gap (1 + 2c) \gap \Lip(\nabla \risk)^2
        \right].
    \end{equation}
\end{theorem}
\begin{proof}
    See \Cref{sec:proof of risk bound smooth}.
\end{proof}
\begin{remark}
    It is immediate to see that taking a constant learning rate proportional to $1 / \sqrt{N}$ in \eqref{eq:risk bound smooth} yields a $O(1/\sqrt{N})$ bound on the excess expected risk; this is akin to the usual bounds for stochastic gradient descent.
\end{remark}

\Cref{thm:risk bound smooth} assumes that $\hstar \in \Chalf(\Hcal)$.
By \Cref{lem:t_i bound}, this is equivalent to
\begin{equation}
    \label{eq:hstar condition}
    \sum_{i=1}^\infty \frac{\dotprod{\hstar, e_i}^2}{t_i} < +\infty.
\end{equation}
Although there is no way to guarantee this holds \emph{a priori} without additional knowledge about $\hstar$, it is possible to show that, whatever $\hstar$ may be, there exists some distribution $\Kdist$ for which \eqref{eq:hstar condition} holds.%
\begin{proposition}
    \label{prop:t_i for hstar}
    For any $\hstar \in \Hcal$, there exists a distribution $\Kdist$ for which $t_i = \Kdist([i, +\infty))$ satisfies \eqref{eq:hstar condition}. For this $\Kdist$, we get $\hstar \in \Chalf(\Hcal)$. %
\end{proposition}
\begin{proof}
    See \Cref{sec:proof of t_i for hstar}.
\end{proof}

\begin{remark}
    We assume throughout that the distribution of $\bz$ given $K=k$ is exactly $\Ncal(\bzero, \bI_k)$, but this is not strictly necessary.
    Our proofs only rely on the first four moments of $\bz$, therefore, all results remain valid if we swap $\Ncal(\bzero, \bI_k)$ by a distribution that matches all of its moments (including mixed ones) up to order four. %
\end{remark}

\section{Using \algoacronym\ to solve PDEs with functional optimization}
\label{sec:pdes}
In this section, we show how our framework can be applied to tackle our motivating example of solving PDEs, presented in Section~\ref{sec:motivation}.
To showcase \algoacronym's flexibility, we consider two linear PDEs --- the heat equation on a bar and the Poisson equation on an L-shaped domain --- and one non-linear PDE, an HJB equation.
They result in two convex and one non-convex optimization problem, respectively.
The HJB equation, for which the assumptions of our convergence results do not hold, offers empirical evidence that our algorithm still provides good approximations in more challenging scenarios.

\subsection{Formal setup and parameters}
Let $\Omega \subset \bbR^d$ be a bounded open set with a boundary smooth enough to admit a trace operator\footnote{We recall that the trace operator is the proper extension of the ``restriction to the boundary'' operator over continuous functions to Sobolev spaces.
It exists, for example, if the boundary of $\Omega$ is Lipschitz~\cite[Theorem~1.5.1.3]{grisvard2011}, in the sense that it is, locally and after possibly a change of coordinates, the graph of a Lipschitz function.}
$ \Trace : \bbH^1(\Omega) \to L^2(\partial \Omega)$.
For the sake of concreteness, we consider general PDEs with Dirichlet boundary conditions:\footnote{Other types of boundary conditions, such as Neumann or periodic boundary conditions, can be treated analogously.}
\begin{equation}
    \label{eq:pde dirichlet}
    \begin{cases}
        L [u] (\bx) = 0 &\quad \bx \in \Omega, \\
        u (\bx) = f (\bx) &\quad \bx \in \Lambda \subset \partial \Omega.
    \end{cases}
\end{equation}
Here, the boundary condition is interpreted in the trace sense and $L$ is a, possibly non-affine, $\ell$-th order differential operator of the form
\begin{equation*}
    L [u] (\bx) = F (\bx, u (\bx), Du (\bx), \ldots, D^{(\ell)} u (\bx)),
\end{equation*}
with $F : \bbR^d \times \bbR \times \bbR^d \times \cdots \times \bbR^{d^\ell} \to \bbR$ being a given function.
We assume that $L[u] \in L^2(\Omega)$ for all $h \in \bbH^\ell (\Omega)$ and define the risk functional
\begin{equation}
    \label{eq:pde risk}
    \risk(h) = \frac{1}{2} \norm*{L[h]}^2_{L^2(\Omega)} + \frac{1}{2} \norm*{(\Trace[h] - f)\ind_{\Lambda}}^2_{L^2(\partial \Omega)},
\end{equation}
as in \Cref{eq:motivating example risk}.
The assumption $L(\bbH^\ell(\Omega)) \subset L^2(\Omega)$ is readily checked if $L$ is an affine operator with bounded coefficients, in which case we furthermore have:
\begin{proposition}
    \label{prop:linear operator risk good}
    If $L : \bbH^\ell(\Omega) \to L^2(\Omega)$ is bounded and affine, then the risk functional \eqref{eq:pde risk} is convex, Fréchet-differentiable and has a Lipschitz gradient.
\end{proposition}
\begin{proof}
    See \Cref{sec:proof of linear operator risk good}.
\end{proof}

To be applied to PDEs, our framework requires a pre-basis $\Bcal = \{ b_i \}_{i\in\bbN}$ for the Sobolev space $\bbH^\ell (\Omega)$, where $\ell \in \bbN$.
To obtain one, we will take advantage of the theory of Reproducing Kernel Hilbert Spaces (RKHSs) and their smoothness properties.
Recall the Matérn kernel~\cite[Section~4.2]{rasmussen2008}, given by %
\begin{equation}
    \label{eq:matern kernel}
    \kernel_{\nu, \eta}(\bx, \by) = \frac{2^{1 - \nu}}{\Gamma(\nu)} \left(
        \frac{\sqrt{2} \nu \norm{\bx - \by}}{\eta}
    \right)^\nu K_\nu \! \left(
        \frac{\sqrt{2} \nu \norm{\bx - \by}}{\eta}
    \right),
    \qquad \bx, \by \in \Omega,
\end{equation}
where $\eta > 0$ is the bandwidth parameter, $\nu > 0$ is the smoothness parameter, $\Gamma$ is the usual 
Gamma function and $K_\nu$ is a modified Bessel function.
When $\nu - 1/2$ is a nonnegative integer, then $\kernel_{\nu, \eta}$ admits a simple closed form as a product of exponential and polynomial terms \cite{rasmussen2008}.

The following proposition establishes that Matérn kernels centered at appropriately chosen centers $\bx$ form a pre-basis for $\bbH^\ell(\Omega)$.
\begin{proposition}
    \label{thm:sobolev pre basis}
    Take $d \geq 2$.
    Let $\Omega \subset \bbR^d$ be a bounded, open set with Lipschitz boundary, let $\superdomain \supset \Omega$
    and let $\{\bx_i\}_{i\in\bbN}$ be a dense subset of $\superdomain$.
    For $\ell \in \bbN$, take $\nu > 0$ such that $\nu + d / 2$ is greater than or equal to $\ell$.
    Then, for any $\eta > 0$, the set $\Bcal = \{ \kernel_{\nu, \eta}(\bx_i, \cdot) \}_{i \in \bbN}$ is a pre-basis for $\bbH^{\ell}(\Omega)$.
\end{proposition}
\begin{proof}
    The proof is given in \Cref{sec:proof of sobolev pre basis}.
\end{proof}
So, as long as we have a countable set of points $\{\bx_i\}_{i\in\bbN}$ that is dense in some set containing $\Omega$ (e.g., generated from quasi-random sampling at rational coordinates), the sequence of Matérn kernels centered at the $\bx_i$ is a suitable pre-basis.
This holds for any differentiability order $\ell$ and any bounded domain $\Omega$ with Lipschitz boundary, even when the Sobolev space $\bbH^\ell(\Omega)$ is not an RKHS and/or the computation of an orthonormal basis becomes intractable.\footnote{The computation of an orthonormal basis is highly dependent on both $\ell$ and $\Omega$, and generally one is not known in closed form.}

\paragraph{Implementation details} In order to compute Sobolev inner products (needed for the QR decomposition of the quasi-matrices), as well as directional derivatives of the risk defined in \eqref{eq:pde risk}, we need to compute multi-dimensional integrals.
In our experiments, we estimate these integrals with Quasi-Monte Carlo using the Roberts sequence \cite{robertsqmc}.
We chose this method for its scalability to higher dimensions, and superior performance compared to simple random sampling.
Computation of the directional derivatives is done using automatic differentiation, by writing $D \risk (h ; v) = \frac{\ddrm}{\ddrm \delta} \risk(h + \delta v) \big|_{\delta = 0}$.
For the distribution $\Kdist$ of our random dimension $K$, we choose a Poisson distribution truncated to powers of two.
More precisely, given a $\mu \in \bbN$, we take $\Kdist = \PoissonPofTwo(\mu)$, where $\PoissonPofTwo(\mu)$ is the distribution of $ 2^{\lceil \log_2 (K + 1) \rceil}$ when $K \sim \Poisson(\mu)$; we do this to avoid excessive code recompilations due to changing array shapes.
A Python/JAX implementation of our algorithm, along with code to reproduce our results, can be found in \texttt{\url{https://github.com/Caioflp/functional-gradient-free-descent}}.

\begin{figure}[t]
    \centering
    \begin{subfigure}{.9\textwidth}
        \includegraphics[width=\textwidth]{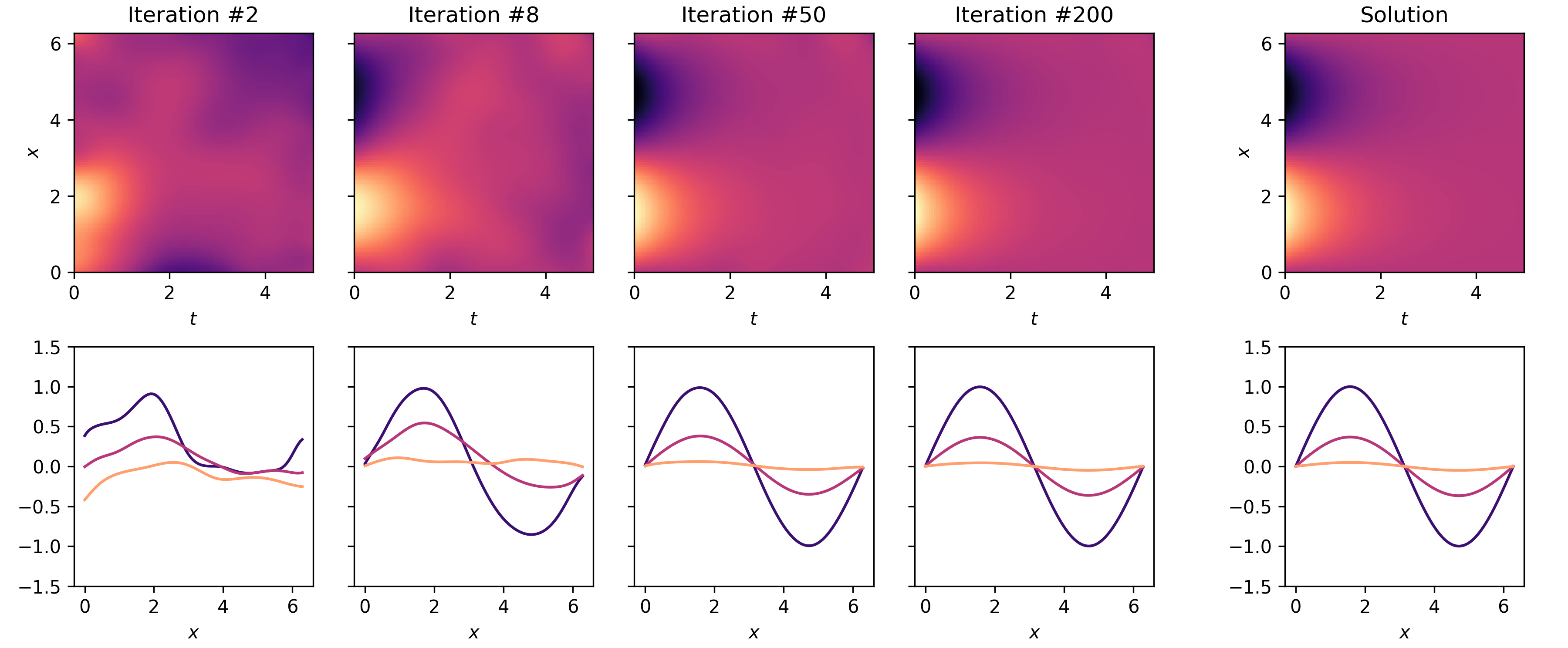}
        \caption{\algoacronym\ with preconditioning ($\lambda_k = t_k$)}
    \end{subfigure}
    \begin{subfigure}{.9\textwidth}
        \includegraphics[width=\textwidth]{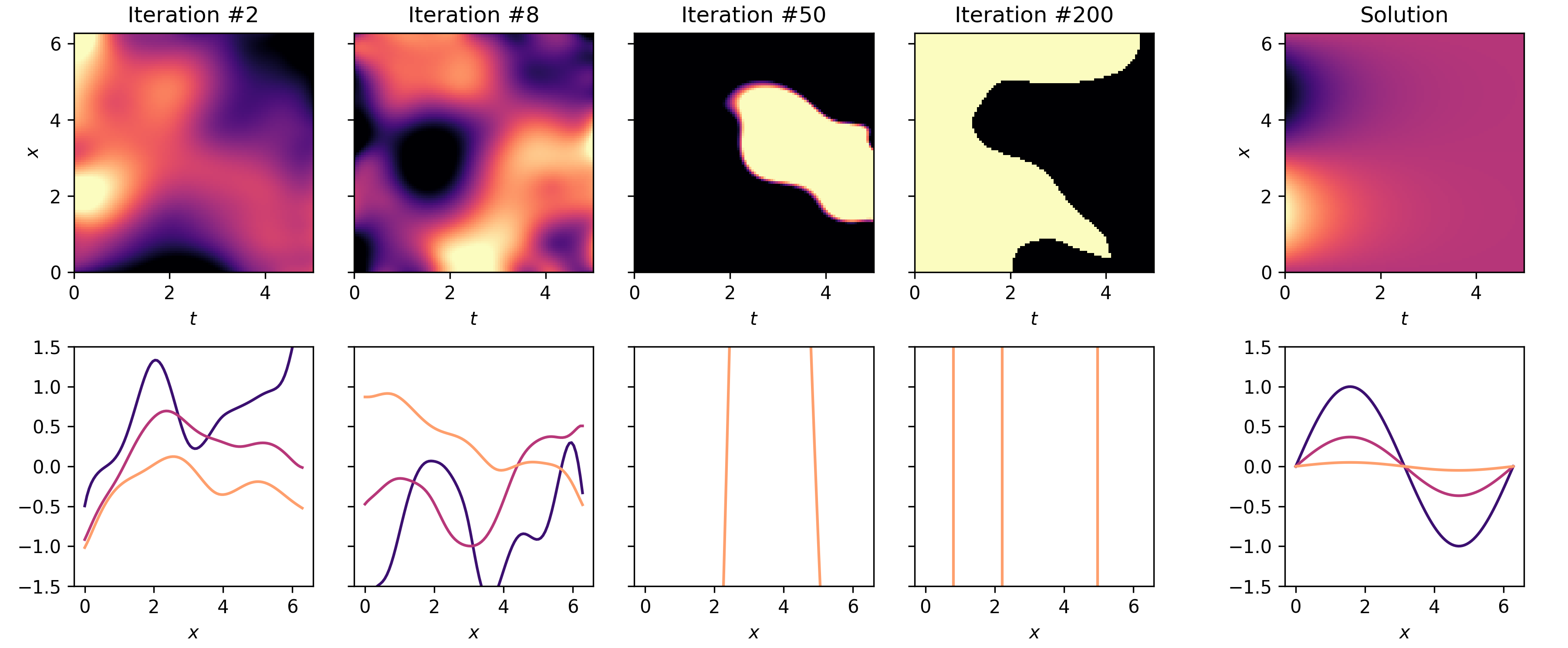}
        \caption{\algoacronym\ without preconditioning ($\lambda_k \equiv 1$)}
    \end{subfigure}
    \caption{
        \textbf{\algoacronym\ on the heat equation.}
        On the top row we visualize a heatmap of the iterates $h_n(t, x)$ and the solution $u(t, x)$, with $x \in [0, 2\pi]$ and $t \in [0, 1]$; on the bottom row, we visualize the solutions $u(t, x)$ at three points in time (\textcolor{magma1}{$t = 0$}, \textcolor{magma2}{$t = 0.5$} and \textcolor{magma3}{$t = 1$}).
        As the theory suggests, our method with preconditioning effectively and provably finds the correct solution to the heat equation.
        Without preconditioning, the procedure quickly diverges.
    }
    \label{fig:heat equation works}
\end{figure}

\subsection{Case I: Heat equation}
As a first example, consider the heat equation
\begin{equation}
    \begin{cases}
        \partial_t u(t, \bx) = \Delta u (t, \bx) &\quad (t, \bx) \in \Omega \defeq (0, T) \times U, \\
        u(t, \bx) = f(t, \bx) &\quad (t, \bx) \in \Lambda \subset \partial \Omega,
    \end{cases}
\end{equation}
for some open set $U \subset \bbR^d$ and boundary condition $f \in L^2(\partial \Omega)$ over the parabolic boundary of $\Omega$, given by $\Lambda = (\{0\} \times U) \cup ((0, T) \times \partial U) $.
When written in the form of \Cref{eq:motivating pde} we see that the corresponding differential operator $L = \partial_t - \Delta$ is linear, rendering the optimization problem convex.
For the sake of concreteness, we shall consider the following special case:
\begin{equation}
    \label{eq:concrete heat equation}
    \begin{cases}
        \partial_t u(t, x) = \Delta u (t, x), &\quad t \in (0, T), x \in (0, 2\pi), \\
        u(0, x) = \sin x, &\quad x \in (0, 2\pi), \\
        u(t, 0) = u(t, 2\pi) = 0, &\quad t \in (0, T).
    \end{cases}
\end{equation}
This has a known unique solution given by $u(t, x) = e^{-t} \sin x$, allowing us to evaluate our procedure.

\begin{figure}[t]
    \centering
    \hfill
    \includegraphics[width=.9\textwidth]{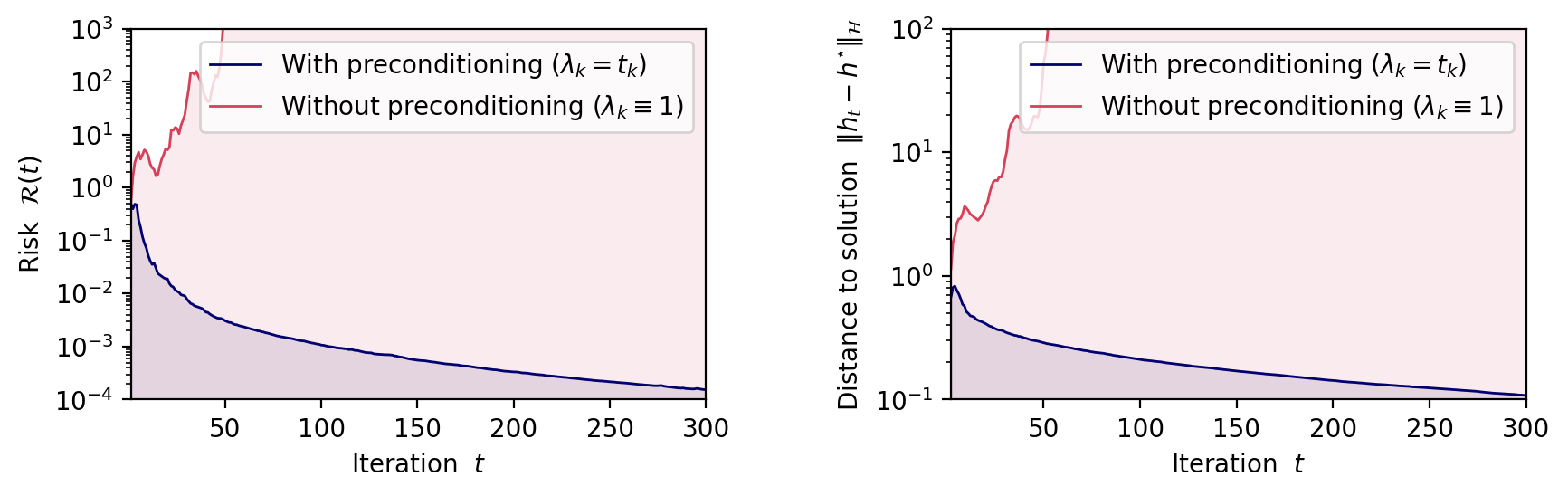}
    \caption{
        \textbf{Convergence of \algoacronym\ on the heat equation.}
        The plots display (left) the evolution of the risk $\risk(h_n)$ over iterations, and (right) the evolution of the distance from the estimates $h_n$ to the optimal solution $h^\star$ in $\bbH^2(\Omega)$-norm.
        Note how our algorithm with preconditioning steadily reduces both the risk and distance to the solution, whereas without preconditioning it quickly diverges.
    }
    \label{fig:heat equation learning curves}
\end{figure}

To solve \Cref{eq:concrete heat equation}, we let $\Omega = (0, T) \times (0, 2 \pi)$ and construct the risk functional \eqref{eq:pde risk} in $\bbH^2(\Omega)$: %
\begin{align*}
    \risk(h)
    &= \frac{1}{2} \int_0^T \int_0^{2\pi} \bigl[ \partial_t h(t, x) - \Delta h(t, x) \bigr]^2 \drm x \drm t
        + \frac{1}{2} \int_0^{2\pi} \bigl[ h(0, x) - \sin x \bigr]^2 \drm x
    \\ &\qquad
        + \frac{1}{2} \int_0^T \bigl[ h(t, 0) - 0 \bigr]^2 \drm t
        + \frac{1}{2} \int_0^T \bigl[ h(t, 2\pi) - 0 \bigr]^2 \drm t.
\end{align*}
Note that, by the linearity of $L$ and the discussion in Section \ref{sec:motivation}, this is a well-defined, convex, Fréchet differentiable risk functional on $\bbH^2(\Omega)$, rendering our theory immediately applicable.

We take $T = 5$, $\Kdist = \PoissonPofTwo(64)$ with sample sizes $M_k = \lceil k / 2 \rceil$, starting point $h_1 \equiv 0$ and constant learning rates $\alpha_n \equiv 0.2$.
For our pre-basis, we take $\eta = 10$ and sample centers $\{\bx_i\}_{i\in\bbN}$ over the scaled domain $2 \cdot \Omega$ with the Roberts quasi-random sequence \cite{robertsqmc}; this ensures that the centered kernels from \Cref{thm:sobolev pre basis} form a proper pre-basis.

\Cref{fig:heat equation works} illustrates our procedure in two settings: with preconditioning (i.e., $\lambda_k = t_k$), and without (i.e., $\lambda_k \equiv 1$).
As we can see, under preconditioning we fall under the regime of \Cref{thm:risk bound,thm:risk bound smooth}, and our method converges to the true solution.
However, without preconditioning, which is equivalent to running a naïve SGD method, our method quickly diverges.
This is most likely caused by the high variance of the stochastic gradients, as noted in \Cref{cor:no precond high var}.
In a similar vein, \Cref{fig:heat equation learning curves} clearly displays the convergence of our preconditioned method to the minimizer, while also exhibiting the quick divergent behavior produced by vanilla stochastic gradients.

\subsection{Case II: Poisson equation over an L-shaped domain}
Consider the Poisson equation:
\begin{equation}
    \label{eq:Poisson L-shaped}
    \begin{cases}
        \Delta u(\bx) = \phi(\bx) &\quad \bx \in \Omega, \\
        u(\bx) = f(\bx) &\quad \bx \in \partial \Omega,
    \end{cases}
\end{equation}
where $\Omega \subset \bbR^d$, $\phi \in L^2(\Omega)$ is given and $f \in L^2(\partial \Omega)$ is a boundary condition.
In the notation of \Cref{eq:motivating pde}, the residual differential operator is given by $L[u] = \Delta u - \phi$.

In this example, we wish to explore our method's performance in a problem whose domain $\Omega$ is not as simple as a box or as a sphere, and thus where obtaining a full orthonormal basis for $\bbH^2(\Omega)$ may be prohibitively hard.
Therefore, we set $d = 2$ and take $\Omega$ to be L-shaped, $\Omega = (-\pi, \pi)^2 \setminus \bigl( [0, \pi] \times [-\pi, 0] \bigr)$, $\phi(x,y) = - (x^2 + y^2) \sin (x y)$ and $f(x,y) = 0$.
Its unique solution is given by $u(x, y) = \sin(x y)$,
and the corresponding risk functional \eqref{eq:pde risk} in $\bbH^2(\Omega)$:
\begin{equation*}
    \begin{aligned}
        \risk(h) 
        &= \frac{1}{2}\int_{\Omega} (\Delta h(x, y) + (x^2 + y^2)\sin(xy))^2 \drm x \ddrm y + \frac{1}{2} \int_{\partial \Omega} h(x, y)^2 \drm x \ddrm y.
    \end{aligned}
\end{equation*}
Since the differential operator $L: \bbH^2(\Omega) \to L^2(\Omega)$ is affine, $\risk$ is a well-defined, Fréchet differentiable, convex functional over $\bbH^2(\Omega)$.

For this experiment, we take $\Kdist = \PoissonPofTwo(128)$ with sample sizes $M_k = \lceil k / 4 \rceil$, starting point $h_1 \equiv 0$ and constant learning rates $\alpha_n \equiv 0.1$.
We also apply the correct preconditioning $\lambda_k = t_k$.
For our pre-basis, we once again take $\eta = 10$, but now we use the Roberts sequence to sample centers $\{\bx_i\}_{i\in\bbN}$ over the larger square $(-2\pi, 2\pi)^2$ containing the L-shaped domain $\Omega$.
\Cref{fig:Poisson equation works and learning curve} depicts the results, which show that our method still converges to the solution of the PDE in spite of the more challenging geometry of $\Omega$ when compared to that of Case I.

\begin{figure}[t]
    \centering
    \includegraphics[width=.9\textwidth]{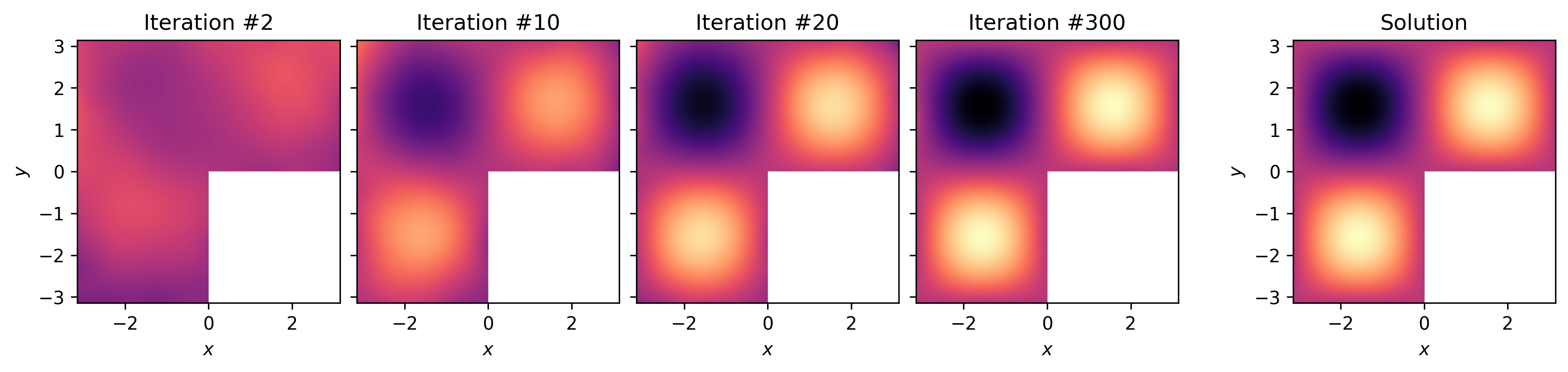}
    \caption{
        \textbf{\algoacronym\ on the Poisson equation in an L-shaped domain.}
        We visualize heatmaps of the solution and our method's estimates over the L-shaped domain.%
        We can see that even though the domain $\Omega$ is not as simple as a box or a sphere, our method still exhibits a solid performance.
    }
    \label{fig:Poisson equation works and learning curve}
\end{figure}

\subsection{Case III: Hamilton-Jacobi-Bellman equation}
For our third example, we consider a $(1+d)$-dimensional HJB equation:
\begin{equation}
    \label{eq:HJB equation}
    \begin{cases}
        \displaystyle
        \partial_t u(t, \bx) + H(\bx, \partial_{\bx} u(t, \bx), \partial_{\bx\bx} u (t, \bx)) = 0,
        &\quad (t, \bx) \in (0, T) \times \bbR^d, \\
        u(T, \bx) = \bx^\trp \bD \bx &\quad \bx \in \bbR^d,
    \end{cases}
\end{equation}
where the Hamiltonian $H: \bbR^d \times \bbR^d \times \bbR^{d^2} \to \bbR$ is given by
\begin{equation}
    \label{eq:hamiltonian}
    \begin{aligned}
        H(\bx, \bp, \bM) &= \bx^\trp \bQ \bx + \bx^\trp \bA^\trp \bp - \frac{1}{4} (\bB^\trp \bp)^\trp \bS^{-1} (\bB^\trp \bp) + \frac{1}{2} \trace (\bC \bC^\trp \bM),
    \end{aligned}
\end{equation}
with optimal control $\bc^{\star} = -\frac{1}{2} \bS^{-1}\bB^{\trp} \bp$. We require the matrices $\bQ, \bS$ and $\bD$ to be symmetric and $\bS$ to also be positive definite.
This particular HJB equation is satisfied by the value function of the following stochastic control problem:
\begin{equation}
    \label{eq:control problem}
    \begin{aligned}
        \inf_{(\bc_s)_{s \in [t, T]}}
        \quad & \mean \left[ \int_t^T \big[ \bx_s^\trp \bQ \bx_s + \bc_s^\trp \bS \bc_s \big] \drm s + \bx_T^\trp \bD \bx_T \right],
    \end{aligned}
\end{equation}
such that $\ddrm \bx_t = [\bA \bx_t + \bB \bc_t] \drm t + \bC \ddrm \bw_t, \quad \bx_0 = \bx$, where $\bw_t$ is a $d$-dimensional standard Brownian motion.
The closed-form solution to this nonlinear, multidimensional PDE is given by
\begin{equation*}
    \ustar (t, \bx) = \bx^\trp \bP(t) \bx + \int_{t}^T \trace(\bC^\trp \bP(s) \bC) \drm s,
\end{equation*}
where $\bP(t)$ solves a well-known Riccati ODE.
In this section we will focus on the case where $\bQ = \bS = \bD = d^{-1} \bI$ and $\bA = \bB = \bC = \bI$, under which the solution is known to be given by $\ustar(t, \bx) = \frac{p(t)}{d}\norm{\bx}^2 + q(t)$, where $p(t) = 1 + \sqrt{2} \tanh(\sqrt{2}(T - t))$ and $q(t) = \int_t^T p(s) \drm s$; the optimal control is given in feedback form by $\bc^{\star}(t, \bx) = - \frac{d}{2} \partial_{\bx} \ustar(t, \bx)$.
We will now describe how to appropriately formulate the risk functional $\risk : \bbH^2 \to \bbR$ for this PDE.

\begin{figure}[t]
    \centering
    \centering
    \includegraphics[width=.9\textwidth]{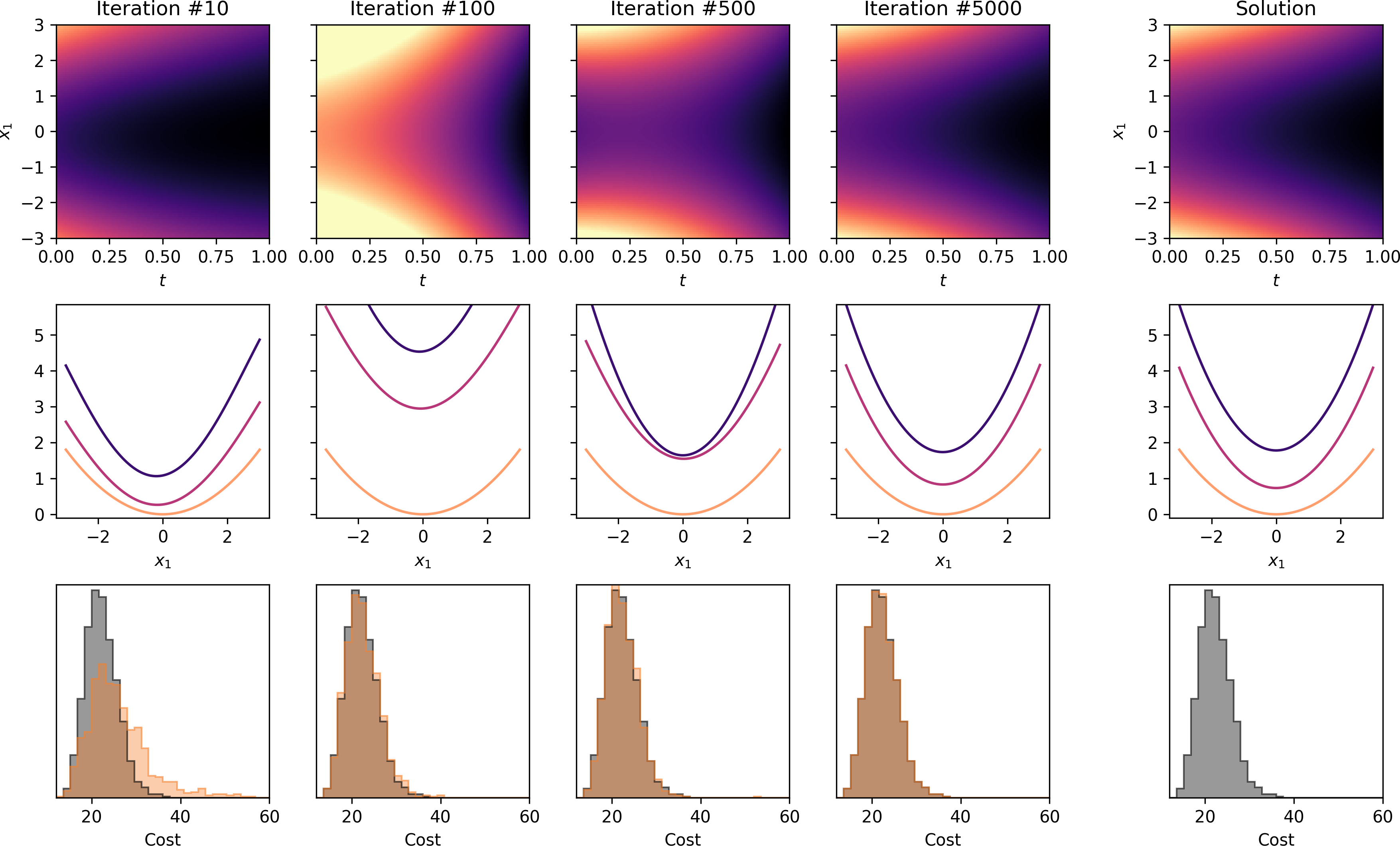}
    \caption[\algoacronym\ on a $(1+5)$-dimensional HJB equation.]{
        \textbf{\algoacronym\ on a $(1+5)$-dimensional HJB equation.}
        The first line shows heatmaps of the solution $\ustar$ and our method's iterates $h_n$ over $(t, \bx_1)$ while $\bx_2, \ldots, \bx_d$ are set to $0$.
        The second line shows slices of the heatmaps for times \textcolor{magma1}{$t = 0$}, \textcolor{magma2}{$t = 0.5$} and \textcolor{magma3}{$t = 1$}.
        The third line shows the cost distributions we obtain when we simulate the control problem \eqref{eq:control problem} using both the optimal control $\bc_s = \bc^{\star}(t, \bx_s)$ and the control $\bc_s = - \frac{d}{2} \partial_{\bx} h_n(t, \bx_s)$, induced by our approximate value functions, with total cost given by $d^{-1}\big[ \int_0^T \norm{\bx_s}^2 + \norm{\bc_s}^2 \drm s  + \norm{\bx_T}^2 \big]$ .
    }
    \label{fig:hjb equation works}
\end{figure}

We start by noting that, although the PDE \eqref{eq:HJB equation} is defined in $(0, T) \times \bbR^d$, in practice we wish to solve it in a bounded subdomain $\Omega = (0, T) \times (-\xbar, \xbar)^d $, where $\xbar > 0$ is a constant.
However, since the PDE is well-defined for any $\xbar > 0$, we formulate the optimization problem over a larger domain $\Omegatilde = (0, T) \times (-\Xbar, \Xbar)^d$, where $\Xbar > \xbar$, and only evaluate the solution over $\Omega$.
This is a common strategy to improve the quality of the solution on the target domain $\Omega$.

Next, we need to adapt the framework described in \Cref{sec:motivation} due to the nonlinearity of the PDE. Because the Hamiltonian \eqref{eq:hamiltonian} is a quadratic function of $\bp = \partial_{\bx} u(t, \bx)$, to ensure the risk is well-defined over $\bbH^2(\Omegatilde)$ for any dimension $d \in \bbN$ we make a slight modification to this nonlinearity when defining the risk.
Let $\psi : \bbR^d \to \bbR$ be the Lipschitz-continuous approximation of $\norm{\cdot}^2$ given by
\begin{equation*}
    \psi(\bp) = \begin{cases}
        \norm{\bp}^2 & \quad \text{if} \quad \norm{\bp} \leq B, \\
        2 B \norm{\bp} - B^2 &\quad \text{otherwise},
    \end{cases}
\end{equation*}
where $B > 0$ is some large constant.
With $\psi$, we define the modified Hamiltonian
\begin{equation*}
    \Htilde(\bx, \bp, \bM) = d^{-1}\norm{\bx}^2 + \bx^\trp \bp - \frac{d}{4} \psi(\bp) + \frac{1}{2} \trace(\bM),
\end{equation*}
and use it to define the functional $\Ecal : \bbH^2(\Omegatilde) \to \bbR$ through
\begin{equation}
    \begin{aligned}
    \Ecal(h) &= \frac{1}{2} \int_{\Omegatilde} \big[ \partial_t h(t, \bx) + \Htilde(\bx, \partial_{\bx}h(t, \bx), \partial_{\bx\bx}h(t, \bx)) \big]^2 \drm t \drm \bx \\
    &\hspace{1.5cm}
    + \frac{1}{2} \int_{(-\Xbar, \Xbar)^d} \big[ h(T, \bx) - d^{-1}\norm{\bx}^2 \big]^2 \drm \bx.
    \end{aligned}
\end{equation}
In practice, for $B$ much larger than $\sup_{(t, \bx) \in \Omegatilde} \norm{\partial_{\bx}\ustar(t, \bx)}$, this modified Hamiltonian still induces a PDE whose solution is $\ustar$, and the functional $\Ecal$ equals its version using $H$ instead of $\Htilde$ for all functions observed in our optimization algorithm. %
Finally, we use $\Ecal$ to define the actual risk functional $\risk: \bbH^2(\Omegatilde) \to \bbR$ in a way that already guarantees a correct boundary condition:
\begin{equation}
    \label{eq:hjb risk}
    \risk(h) = \Ecal(\EnforceBC(h)),
\end{equation}
where $\EnforceBC : \bbH^2(\Omegatilde) \to \bbH^2(\Omegatilde)$ is given by
\begin{equation}
    \label{eq:bc operator}
    \EnforceBC(h)(t, \bx) = (T - t) h(t, \bx) + d^{-1}\norm{\bx}^2.
\end{equation}
This last step is standard in optimization-based PDE solvers \cite{cohen2026neuralactorcriticmethodshamiltonjacobibellman, cohen2023neuralqlearning, cohen2026globalconvdgmsiamfinancialmath}, being a simple way to improve performance, particularly in high-dimensional scenarios.
\begin{proposition}
    \label{prop:hjb risk properties}
    The functional defined by \Cref{eq:hjb risk} is a well-defined, Fréchet-differentiable, nonconvex functional over $\bbH^2(\Omegatilde)$.
\end{proposition}
\begin{proof}
    See \Cref{sec:proof of hjb risk properties}.
\end{proof}

For this experiment we take $T = 1, d = 5, \xbar = 3$, $\Xbar = 5$, $\Kdist = \PoissonPofTwo(1024)$ with sample sizes $M_k = \lceil k / 2 \rceil$, starting point $h_1 \equiv 0$, constant learning rates $\alpha_n \equiv 0.02$,
and use the same Matérn-based pre-basis with $\eta = 10$ and centers sampled from the Roberts sequence over $20 \cdot \Omegatilde$.
We furthermore use the proper preconditioning $\lambda_k = t_k$.
\Cref{fig:hjb equation works} shows the results of this experiment.
We see that despite the risk being nonconvex, our method still effectively converges to the correct solution, highlighting its versatility.

\section{Conclusion}

In this paper, we introduced a random gradient-free method for optimization in infinite-dimensional separable Hilbert spaces, avoiding the computational challenges common to functional gradient descent.
By using only directional derivatives and a pre-basis, our approach offers a practical and broadly applicable framework for solving functional optimization problems.
Our work expands the applicability of functional gradient methods, particularly to settings with specialized losses and for which functional gradients would need to be computed accordingly.
Future work will consider variants such as constrained, mirror and accelerated versions of our algorithm.

\section*{Acknowledgments}
We would like to thank Alfredo Iusem and Fábio Ramos for the helpful suggestions and comments.
We disclose the use of LLMs as minor writing assistance tools, and we naturally assume responsibility for all content.

\printbibliography

\appendix

\section{Proofs of selected results}

\subsection{Proof of \Cref{prop: unbiased precond gradient}}
\label{sec:proof of unbiased precond gradient}

The fact that $\mean [v \mid K] = 0$ is obvious from \Cref{eq:v definition}, since $\mean [\bz \mid K] = \bzero$ and then
\begin{equation*}
    \mean[v \mid K] = B_K^{\vphantom{-1}} \bR_K^{-1} \bT_K^{-\frac{1}{2}} \mean[\bz \mid K] = 0.
\end{equation*}

To show that $\mean[v] = 0$, we need to show that $\mean[\dotprod{v, h}] = 0$ for any $h \in \Hcal.$
This is obtained if we use the tower property of conditional expectations to compute $\mean[\dotprod{v, h}] = \mean[\mean[\dotprod{v, h} \mid K]] = \mean[\dotprod{\mean[v \mid K], h}] = 0$.
However, to make sure that the conditional expectations we are using actually exist and we can use the tower property, we must first ensure that $\dotprod{v, h} \in L^1(\Scal, \Fcal, \prob)$.
This is done in the following lemma:
\begin{lemma}
    \label{lem:integrability for v mean}
    For any $h \in \Hcal$, $\dotprod{v, h} \in L^1(\Scal, \Fcal, \prob)$.
\end{lemma}
\begin{proof}
    Since the random variable $\abs{\dotprod{v, h}}$ is non-negative, its conditional expectation given $K$ always exists, even if $\abs{\dotprod{v, h}}$ is not integrable.
    Therefore, \emph{now} we can use the tower property and write:
    \begin{equation*}
        \mean[\abs{\dotprod{v, h}}]
        = \mean[\mean [ \abs{\dotprod{v, h}} \mid K ]]
        \leq \mean \left[\mean[\dotprod{v, h}^2 \mid K]^{\frac{1}{2}}\right]
        = \mean \left[\mean[\dotprod{Q_K^{\vphantom{-1}} \bT_K^{-\frac{1}{2}} \bz, h}^2 \mid K]^{\frac{1}{2}}\right],
    \end{equation*}
    where we have used Jensen's inequality and the identity $v = Q_K^{\vphantom{-1}}\bT_K^{-\frac{1}{2}} \bz$, which is immediate from \eqref{eq:v definition}.
    Now note that, given $K$, we have
    \begin{equation*}
        \dotprod{Q_K^{\vphantom{-1}} \bT_K^{-\frac{1}{2}}\bz, h} = \dotprod{\bz, \bT_K^{-\frac{1}{2}} Q_K^* h} \sim \Ncal(0, \norm*{\bT_K^{-\frac{1}{2}} Q_K^{*\vphantom{-1}} h}^2).
    \end{equation*}
    This implies:
    \begin{equation*}
        \begin{aligned}
        \mean \left[\mean[\dotprod{Q_K^{\vphantom{-1}} \bT_K^{-\frac{1}{2}} \bz, h}^2 \mid K]^{\frac{1}{2}}\right]
        &= \mean \left[
            \norm*{\bT_K^{-\frac{1}{2}} Q_K^{*\vphantom{-1}} h}
        \right] \\
        &\leq \mean \left[
            \norm*{\bT_K^{-\frac{1}{2}} Q_K^{*\vphantom{-1}} h}^2
        \right]^{\frac{1}{2}}
        = \mean \left[
            \sum_{i=1}^K \frac{\dotprod{h, e_i}^2}{t_i}
        \right]^{\frac{1}{2}}.
        \end{aligned}
    \end{equation*}
    Now we rewrite the summation and apply Tonelli's Theorem to obtain
    \begin{align*}
        \mean \left[
            \sum_{i=1}^K \frac{\dotprod{h, e_i}^2}{t_i}
        \right]^{\frac{1}{2}}
        &=  \mean \left[
            \sum_{i=1}^\infty \frac{\ind\{ K \geq i \} \dotprod{h, e_i}^2}{t_i}
        \right]^{\frac{1}{2}} \\
        &= \left(
            \sum_{i=1}^\infty \mean \left[ \frac{\ind\{ K \geq i \} \dotprod{h, e_i}^2}{t_i} \right]
        \right)^{\frac{1}{2}}
        = \left(
            \sum_{i=1}^\infty \dotprod{h, e_i}^2
        \right)^{\frac{1}{2}}
        = \norm{h}.
    \end{align*}
    Thus, $\mean[\abs{\dotprod{v, h}}] \leq \norm{h} < +\infty$.
\end{proof}

Since the distribution of $v$ given $K$ has zero mean and is finite dimensional, we have $\Cov[v \mid K] = \mean [v \otimes v \mid K]$.
We expand this expression, using again that $v = Q_K^{\vphantom{-1}} \bT^{-\frac{1}{2}}_K \bz$:
\begin{align*}
    \mean [v \otimes v \mid K]
    &= \mean [Q^{\vphantom{-\frac{1}{2}}}_K \bT_K^{-\frac{1}{2}} \bz \otimes Q^{\vphantom{-\frac{1}{2}}}_K \bT_K^{-\frac{1}{2}} \bz \mid K] \\
    &= \mean [Q^{\vphantom{-\frac{1}{2}}}_K \bT_K^{-\frac{1}{2}} (\bz \otimes \bz) \bT_K^{-\frac{1}{2}} Q^{\vphantom{-\frac{1}{2}}*}_K \mid K]
    = Q^{\vphantom{-\frac{1}{2}}}_K \bT_K^{-\frac{1}{2}} \mean [ \bz \otimes \bz \mid K ] \bT_K^{-\frac{1}{2}} Q^{\vphantom{-\frac{1}{2}}*}_K.
\end{align*}
By the definition of $\bz$ in \Cref{eq:v definition}, we have $\mean [\bz \otimes \bz \mid K] = \bI_K$.
Thus: %
\begin{equation*}
   \mean [v \otimes v \mid K] 
   = Q^{\vphantom{-\frac{1}{2}}}_K \bT_K^{-1} Q^{\vphantom{-\frac{1}{2}}*}_K.
\end{equation*}
This shows that $\Cov[\vdistgivenK] = \Cov[v \mid K] = Q^{\vphantom{-\frac{1}{2}}}_K \bT_K^{-1} Q^{\vphantom{-\frac{1}{2}}*}_K$.

To obtain $\Cov[\vdist] = \Cov[v] = I$, we must take $h, h' \in \Hcal$ and show that $\mean[\dotprod{h, v} \dotprod{h', v}] = \dotprod{h, h'}$.
Once again, we would like to proceed by conditioning on $K$.
The next lemma shows that this is possible:
\begin{lemma}
    \label{lem:integrability for v cov}
    For any $h, h' \in \Hcal$, we have $\dotprod{h, v}\dotprod{h', v} \in L^1(\Scal, \Fcal, \prob)$.
\end{lemma}
\begin{proof}
    By Cauchy-Schwarz in $L^2(\Scal, \Fcal, \prob)$, we have
    \begin{equation*}
        \mean [\abs{\dotprod{h, v} \dotprod{h', v}}]
        \leq \mean[\dotprod{h, v}^2]^{\frac{1}{2}} \mean[\dotprod{h', v}^2]^\frac{1}{2}.
    \end{equation*}
    By the arguments in the proof of \Cref{lem:integrability for v mean} we know that
    \begin{equation*}
        \mean[\dotprod{h', v}^2]
        = \mean\left[\mean[\dotprod{h', v}^2 \mid K]\right]
        = \mean[ \norm*{\bT_K^{-\frac{1}{2}} Q_K^{*\vphantom{-1}} h}^2 ]
        = \norm{h}^2,
    \end{equation*}
    and likewise for $h'$.
    Therefore, $\mean [\abs{\dotprod{h, v} \dotprod{h', v}}] \leq \norm{h} \norm{h'} < +\infty$.
\end{proof}
Now, conditioning on $K$ we obtain
\begin{equation}
    \label{eq:cov v given k expansion}
    \begin{aligned}
    \mean [\dotprod{h, v} \dotprod{h', v}]
    &= \mean [ \mean[\dotprod{h, v} \dotprod{h', v} \mid K] ] \\
    &= \mean [ \dotprod{\Cov[v \mid K]h, h'} ]
    = \mean [ \dotprod{Q^{\vphantom{-\frac{1}{2}}}_K \bT_K^{-1} Q^{\vphantom{-\frac{1}{2}}*}_K h, h'} ]
    = \mean[ \dotprod{\bT_K^{-1} Q^{\vphantom{-1}*}_K h, Q^{\vphantom{-1}*}_Kh'}].
    \end{aligned}
\end{equation}
Since $Q_K = [e_1, \ldots, e_K]$ and the $e_i$ are orthonormal, we have
\begin{equation}
    \label{eq:Q_K inner product expansion}
    \dotprod{\bT_K^{-1} Q^{\vphantom{-1}*}_K h, Q^{\vphantom{-1}*}_Kh'}
    = \sum_{i=1}^K \frac{\dotprod{h, e_i} \dotprod{h', e_i}}{t_i}
    = \sum_{i=1}^\infty \frac{\ind\{ K \geq i \} }{t_i} \dotprod{h, e_i} \dotprod{h', e_i}.
\end{equation}
Taking expectations on both sides and applying Fubini to bring out the sum, we get
\begin{equation}
    \label{eq:fubini}
    \mean [\dotprod{h, v} \dotprod{h', v}]
    = \sum_{i=1}^\infty \frac{\mean[ \ind\{ K \geq i \} ]}{t_i} \dotprod{h, e_i} \dotprod{h', e_i}
    = \sum_{i=1}^\infty \dotprod{h, e_i} \dotprod{h', e_i}
    = \dotprod{h, h'}.
\end{equation}
Therefore, $\Cov[v] = I$.
Finally, to show $\mean_n[\ghat_n(K)] = Cg_n$, we will take an arbitrary $h \in \Hcal$ and show that $\mean_n[\dotprod{\ghat_n(K), h}] = \dotprod{C g_n, h}$.
The argument to show that $\dotprod{\ghat_n(K), h} \in L^1(\Scal, \Fcal, \prob)$ is very similar to the one presented in \Cref{lem:integrability for v cov}, so it will be omitted.
Start by conditioning on $K$ and recalling that $D\risk(h_n, v_m) = \dotprod{g_n, v_m}$:
\begin{align*}
    \mean_n[\dotprod{\ghat_n(K), h}]
    = \mean_n[\mean_n[\dotprod{\ghat_n(K), h} \mid K]]
    &= \mean_n \left[
        \frac{\lambda_K}{M_K} \sum_{m=1}^{M_K} 
        \mean_n [ \dotprod{g_n, v_m} \dotprod{h, v_m} \mid K ]
    \right] \\
    &= \mean_n \left[
        \lambda_K \mean_n [ \dotprod{g_n, v} \dotprod{h, v} \mid K ]
    \right],
\end{align*}
where we have used that $v_1, \ldots, v_{M_K}$ are i.i.d.\ with the same distribution as $v$.
By Equations \eqref{eq:cov v given k expansion} and \eqref{eq:Q_K inner product expansion} we know that
\begin{equation}
    \mean_n [ \dotprod{g_n, v} \dotprod{h, v} \mid K ]
    = \sum_{i=1}^\infty \frac{\ind \{ K \geq i \} }{t_i} \dotprod{g_n, e_i} \dotprod{h, e_i}.
\end{equation}
Now, again by Fubini we obtain
\begin{align*}
    \mean_n[\dotprod{\ghat_n(K), h}]
    &= \mean_n \left[
        \sum_{i=1}^\infty \frac{\lambda_K \ind \{ K \geq i \}}{t_i} \dotprod{g_n, e_i} \dotprod{h, e_i}.
    \right] \\
    &= \sum_{i=1}^\infty \mean \left[ \lambda_K \mid K \geq i \right] \dotprod{g_n, e_i} \dotprod{h, e_i}
    = \dotprod{C g_n, h},
\end{align*}
which finishes the proof.

\subsection{Proof of \Cref{thm:starting bound}}
\label{sec:proof starting bound}
Let us start by recalling that $\ghat_n \in \bigcup_{k=1}^\infty \prebasis_k \subset C(\Hcal)$ almost surely.
Thus, since $h_1 \in C(\Hcal)$ by assumption, we have $h_n \in C(\Hcal)$ for all $n = 1, \ldots, N$.
We emphasize that both $\ghat_n$ and $h_n$ belong to $C(\Hcal)$ in addition to $\Chalf(\Hcal)$ because we will use the following property of $\dotprod{\cdot, \cdot}_C$: if $h \in \Chalf(\Hcal)$ and $h' \in C(\Hcal)$, we have $\dotprod{h, h'}_C = \dotprod{h, \Cinv h'}$.

Set $g_n \defeq \nabla \risk (h_n)$ and expand $\frac{1}{2}\norm*{h_{n+1} - h}_{C}^2$:
\begin{align*}
    \frac{1}{2} \norm*{h_{n+1} - h}_{C}^2
    &= \frac{1}{2}\norm*{h_n - \alpha_n \ghat_{n} - h}_{C}^2 \\
    &= \frac{1}{2}\norm*{h_n - h}_{C}^2 + \frac{1}{2}\alpha_n^2 \norm*{\ghat_{n}}_{C}^2 - \alpha_n \dotprod{h_n - h, C^{-1}\ghat_n} \\
    &= \frac{1}{2}\norm*{h_n - h}_{C}^2 + \frac{1}{2}\alpha_n^2 \norm*{\ghat_{n}}_{C}^2
    - \alpha_n \dotprod{h_n - h, g_{n}} \\
    &\hspace{1.5cm}
    + \alpha_n \dotprod{h_n - h, g_{n} - C^{-1}\ghat_{n}}.
\end{align*}
By the convexity of $\risk$, we have $\risk(h) - \risk(h_n) \geq \dotprod{h - h_n, g_{n}}$,
which gives us
\begin{equation*}
    - \alpha_n \dotprod{h_n - h, g_{n}} \leq - \alpha_n (\risk(h_n) - \risk(h)).
\end{equation*}
Hence, we obtain
\begin{equation*}
    \begin{aligned}
        \frac{1}{2} \norm*{h_{n+1} - h}_{C}^2
        &\leq \frac{1}{2}\norm*{h_n - h}_{C}^2 + \frac{1}{2}\alpha_n^2 \norm*{\ghat_{n}}_{C}^2
        - \alpha_n (\risk(h_n) - \risk(h) ) \\
        &\hspace{1.3cm}
        + \alpha_n \dotprod{h_n - h, g_{n} - C^{-1}\ghat_{n}},
    \end{aligned}
\end{equation*}
which we rewrite as
\begin{equation*}
    \begin{aligned}
        & \alpha_n (\risk(h_n) - \risk(h)) \\
        &\hspace{1cm}
        \leq \frac{1}{2}\norm*{h_n - h}_{C}^2 + \frac{1}{2}\alpha_n^2 \norm*{\ghat_{n}}_{C}^2
        - \frac{1}{2} \norm*{h_{n+1} - h}_{C}^2 + \alpha_n \dotprod{h_n - h, g_{n} - C^{-1}\ghat_{n}}.
    \end{aligned}
\end{equation*}
Summing this inequality for $n = 1, \ldots, N$ we get
\begin{equation*}
    \sum_{n=1}^N \alpha_n (\risk(h_n) - \risk(h)) \leq \frac{1}{2} \norm*{h_1 - h}_{C}^2 + \frac{1}{2} \sum_{n = 1}^N \alpha_n^2 \norm*{\ghat_n}_{C}^2 + \sum_{n=1}^N \alpha_n \dotprod{h_n - h, g_n - C^{-1}\ghat_n}.
\end{equation*}
Again by convexity, we have $\risk(\hat{h}) \leq ( \sum_{n=1}^N \alpha_n \risk(h_n) ) / \sum_{n=1}^N \alpha_n$, implying
\begin{equation*}
    \begin{aligned}
        \risk(\hhat)
        - \risk(h)
        \leq
        \frac{
            \frac{1}{2}\norm*{h - h_1}_{C}^2
        }{
            \sum_{n=1}^N \alpha_n
        }
        + \frac{
            \frac{1}{2}\sum_{n=1}^N \alpha_n^2 \norm*{\ghat_n}_{C}^2
        }{\sum_{n=1}^N \alpha_n}
        + \frac{
            \sum_{n=1}^N \alpha_n \dotprod{h_n - h, g_n - C^{-1}\ghat_n}
        }{
            \sum_{n=1}^N \alpha_n
        }.
    \end{aligned}
\end{equation*}
Finally, we take expectations on both sides, hoping that the last term on the RHS will vanish.
Although we will show this to be true, the reason is \emph{not} because $\mean[C^{-1}\ghat_n] = C^{-1}\mean[\ghat_n]$.
Recall that $C^{-1}$ is a possibly unbounded operator, meaning that it does not necessarily commute with expectations, so we must proceed with care.
The result we will use is the following:
\begin{lemma}
    \label{lem:Cinv unbiasedness}
    If $\mean_n[\norm*{\ghat_n}_C^2] < +\infty$, then for any $h' \in \Chalf(\Hcal)$ we have
    \begin{equation*}
        \mean_n[\dotprod{h', C^{-1}\ghat_n}] = \dotprod{h', g_n}.
    \end{equation*}
\end{lemma}
\begin{proof}
    Analogously to the proof of \Cref{prop: unbiased precond gradient}, we will start by showing that $\dotprod{h', C^{-1}\ghat_n} \in L^1(\Scal, \Fcal, \prob_n)$, where $\prob_n$ is the conditional measure given $h_1, \ldots, h_n$.
    Expanding $h'$ and $\ghat_n$ over the basis $\basis$ we get
    \begin{align*}
        \mean_n [\abs*{\dotprod{h', \Cinv \ghat_n}}]
        &= \mean_n \left[ \abs{
            \sum_{i=1}^\infty \frac{\dotprod{h', e_i} \dotprod{\ghat_n, e_i}}{\gamma_i}
        } \right]
        \leq \mean_n \left[ \sum_{i=1}^\infty \frac{\abs{\dotprod{h', e_i}}}{\sqrt{\gamma_i}} \frac{\abs{\dotprod{\ghat_n, e_i}}}{\sqrt{\gamma_i}} \right].
    \end{align*}
Now, use Cauchy-Schwarz over the space of square-summable sequences to obtain
\begin{equation*}
    \sum_{i=1}^\infty \frac{\abs{\dotprod{h', e_i}}}{\sqrt{\gamma_i}} \frac{\abs{\dotprod{\ghat_n, e_i}}}{\sqrt{\gamma_i}}
    \leq \left(
        \sum_{i=1}^\infty \frac{\dotprod{h', e_i}^2}{\gamma_i}
    \right)^{\frac{1}{2}}
    \left(
        \sum_{i=1}^\infty \frac{\dotprod{\ghat_n, e_i}^2}{\gamma_i}
    \right)^{\frac{1}{2}}
    = \norm*{h'}_C \norm*{\ghat_n}_C.
\end{equation*}
Therefore, 
\begin{equation*}
    \mean_n [\abs*{\dotprod{h', \Cinv \ghat_n}}]
    \leq \norm*{h'}_C \mean_n[{\norm*{\ghat_n}_C}]
    \leq \norm*{h'}_C \mean_n\big[{\norm*{\ghat_n}^2_C}\big]^{\frac{1}{2}} < +\infty,
\end{equation*}
by assumption.
Thus, we can condition on $K$ and use the fact that $\ghat_n \in \prebasis_K$ to write
\begin{align*}
    \mean_n[\dotprod{h', \Cinv \ghat_n}]
    &= \mean_n\!\!\left[
        \mean_n\!\!\left[
            \sum_{i=1}^\infty\frac{\dotprod{h', e_i} \dotprod{\ghat_n, e_i}}{\gamma_i}
            \mid K
        \right]
    \right] \\
    &= \mean_n\!\!\left[
        \mean_n\!\!\left[
            \sum_{i=1}^K\frac{\dotprod{h', e_i} \dotprod{\ghat_n, e_i}}{\gamma_i} \mid K
        \right]
    \right]
    = \mean_n\!\!\left[
            \sum_{i=1}^K\frac{\dotprod{h', e_i} \mean_n[\dotprod{\ghat_n, e_i} \mid K]}{\gamma_i}
    \right]\!\!.
\end{align*}
Now recall that
\begin{equation*}
    \mean_n[\dotprod{\ghat_n, e_i} \mid K]
    = \lambda_K \dotprod{\bT_K^{-1} Q_K^{*\vphantom{-1}}g_n, Q_K^{*\vphantom{-1}}e_i}
    = \frac{\lambda_K \ind\{ K \geq i \}}{t_i} \dotprod{g_n, e_i}.
\end{equation*}
Thus, we have
\begin{equation*}
    \mean_n[\dotprod{h', \Cinv \ghat_n}]
    = \mean_n \left[
        \sum_{i=1}^{\infty} \frac{\lambda_K \ind\{ K \geq i \}}{t_i \gamma_i} \dotprod{h', e_i} \dotprod{g_n, e_i}
    \right].
\end{equation*}
Using Fubini and the definition of $\gamma_i$ as $\mean[\lambda_K \mid K \geq i]$, we obtain
\begin{equation*}
    \mean_n[\dotprod{h', \Cinv \ghat_n}]
    = \sum_{i=1}^\infty \dotprod{h', e_i} \dotprod{g_n, e_i} = \dotprod{h', g_n},
\end{equation*}
which finishes the proof of \Cref{lem:Cinv unbiasedness}.
\end{proof}

Going back to the proof of \Cref{thm:starting bound}, since $h \in \Chalf(\Hcal)$ by assumption, we have $h_n - h \in \Chalf(\Hcal)$ for all $n$, letting us apply \Cref{lem:Cinv unbiasedness} to conclude that
\begin{equation*}
    \mean [\dotprod{h_n - h, g_n - \Cinv \ghat_n}]
    = \mean [\mean_n[\dotprod{h_n - h, g_n - \Cinv \ghat_n}]]
    = 0.
\end{equation*}

\subsection{Proof of \Cref{thm:ghat variance}}
\label{sec:proof variance}
To simplify the notation in this proof, we will drop the subscript $n$ and write $\ghat$ instead of $\ghat_n$.
Start by using the symmetry of $\Cminushalf$, Tonelli and conditioning on $K$ to obtain
\begin{equation}
    \label{eq:first expansion}
    \mean [ \norm*{\ghat}^2_C ]
    = \mean \left[
        \sum_{i=1}^\infty \dotprod{\Cminushalf \ghat, e_i}^2
    \right]
    = \sum_{i=1}^\infty
    \mean \left[
        \frac{\lambda_K^2}{\mean[\lambda_K \mid K \geq i]} \mean[ \dotprod{\gbar, e_i}^2 \mid K ]
    \right],
\end{equation}
where we define
\begin{equation}
    \label{eq:gbar definition}
    \gbar \defeq \lambda_K^{-1} \ghat = \frac{1}{M_K} \sum_{m=1}^{M_K} \dotprod{g, v_m} v_m.
\end{equation}
Our goal now is to compute $\mean[ \dotprod{\gbar, e_i}^2 \mid K ]$.
Expanding the square and using the fact that $v_1, \ldots, v_{M_K}$ are i.i.d.\ given $K$, one may verify that
\begin{align}
    \mean [\dotprod{\gbar, e_i}^2 \mid K]
    &= \mean \left[
        \Bigg\langle
            \frac{1}{M_K} \sum_{m=1}^{M_K}
            \dotprod{g, v_m} v_m,
            e_i
        \Bigg\rangle^2
        \mid K
    \right] \nonumber \\
    &= \frac{1}{M_K} \mean[\dotprod{g, v}^2 \dotprod{v, e_i}^2 \mid K]
    + \left( 1 - \frac{1}{M_K} \right) \frac{\ind \{ K \geq i \}}{t_i^2} \dotprod{g, e_i}^2. \label{eq:second expansion}
\end{align}
To compute $\mean [\dotprod{g, e_i}^2 \dotprod{v, e_i}^2 \mid K]$ we start by rewriting:
\begin{align*}
    \dotprod{g, e_i}^2 \dotprod{v, e_i}^2
    &= \dotprod{
        (v \otimes v) (g \otimes g) (v \otimes v) e_i, e_i
    } \\
    &= \dotprod{
        (Q\bT^{-\frac{1}{2}} \bz \otimes Q \bT^{-\frac{1}{2}} \bz) (g \otimes g) (Q\bT^{-\frac{1}{2}} \bz \otimes Q \bT^{-\frac{1}{2}} \bz) e_i, e_i
    } \\
    &= \dotprod{
        Q \bT^{-\frac{1}{2}} (\bz \otimes \bz) (\bT^{-\frac{1}{2}} Q^* g \otimes \bT^{-\frac{1}{2}} Q^* g) (\bz \otimes \bz) Q^* \bT^{-\frac{1}{2}} e_i, e_i
    } \\
    &= \dotprod{
        (\bz \otimes \bz) \bL (\bz \otimes \bz) \bx, \bx
    }_{\bbR^K},
\end{align*}
where $\bL = \bT^{-\frac{1}{2}} Q^* g \otimes \bT^{-\frac{1}{2}} Q^* g \in \bbR^{K \times K}$ and $\bx = \bT^{-\frac{1}{2}} Q^* e_i \in \bbR^K$.
Hence,
\begin{equation*}
    \mean [ \dotprod{g, e_i}^2 \dotprod{v, e_i}^2 \mid K ]
    = \dotprod{
        \mean \left[
            (\bz \otimes \bz) \bL (\bz \otimes \bz) \mid K
        \right]
        \bx, \bx
    }.
\end{equation*}
Now, since $\bz \sim \Ncal (\bzero, \bI_K)$ and $\bL$ is symmetric, one may show that the following fourth-moment identity holds:
\begin{equation*}
    \mean \left[
        (\bz \otimes \bz) \bL (\bz \otimes \bz) \mid K
    \right]
    = (\trace \bL) \bI_K + 2 \bL.
\end{equation*}
This implies:
\begin{equation*}
    \mean [ \dotprod{g, e_i}^2 \dotprod{v, e_i}^2 \mid K ]
    = \dotprod{
        ( (\trace \bL) \bI_K + 2 \bL ) \bx, \bx
    }_{\bbR^K}
    = (\trace \bL) \norm{\bx}_{\bbR^K}^2 + 2 \dotprod{\bL \bx, \bx}_{\bbR^K}.
\end{equation*}
Substituting $\bL$ and $\bx$, we obtain
\begin{equation}
    \mean [ \dotprod{g, e_i}^2 \dotprod{v, e_i}^2 \mid K ]
    = \left(
        \sum_{j=1}^K \frac{\dotprod{g, e_j}^2}{t_j}
    \right) \frac{\ind \{ K \geq i \}}{t_i} + 2 \frac{\ind \{ K \geq i \}}{t_i^2} \dotprod{g, e_i}^2.
\end{equation}
Substituting in \eqref{eq:second expansion}, we find that
\begin{equation}
    \mean [\dotprod{\gbar, e_i}^2 \mid K]
    = \frac{1}{M_K} \left(
        \sum_{j=1}^K \frac{\dotprod{g, e_j}^2}{t_j}
    \right) \frac{\ind \{ K \geq i \}}{t_i}
    + \left( 1 + \frac{1}{M_K} \right) \frac{\ind \{ K \geq i \}}{t_i^2} \dotprod{g, e_i}^2.
\end{equation}
Now we simply substitute back in \eqref{eq:first expansion} and rearrange summations to obtain \eqref{eq:ghat variance}.

\subsection{Proof of \Cref{thm:risk bound smooth}}
\label{sec:proof of risk bound smooth}

Start by expanding
\begin{equation*}
    \norm{h_{n+1} - \hstar}_C^2 = \norm{h_n - \hstar}_C^2 - 2 \alpha_n \dotprod{h_n - \hstar, \ghat_n}_C + \alpha_n^2 \norm{\ghat_n}_C^2.
\end{equation*}
Take the conditional expectation $\mean_n$ on both sides and apply both \Cref{lem:Cinv unbiasedness} and \Cref{cor:variance bound} to get
\begin{equation}
    \label{eq:intermidiate smooth}
    \mean_n[\norm{h_{n+1} - \hstar}_C^2] \leq \norm{h_n - \hstar}_C^2 - 2 \alpha_n \dotprod{h_n - \hstar, g_n} + 2 \alpha_n^2 (1 + 2c)\norm{g_n}^2.
\end{equation}
Because $\risk$ is convex and has Lipschitz gradients, we have the following co-coercivity inequality \cite[Lemma~2.29]{garrigos2024handbookconvergencetheoremsstochastic}:
\begin{equation}
    \label{eq:co-coercivity}
    \frac{1}{\Lip(\nabla \risk)} \norm{\nabla \risk(h) - \nabla \risk(h')}^2 \leq \dotprod{\nabla \risk(h) - \nabla \risk(h'), h - h'},
\end{equation}
for any $h, h' \in \Hcal$.
Applying \eqref{eq:co-coercivity} to \eqref{eq:intermidiate smooth} with $h = h_n$ and $h' = \hstar$ and recalling that $\nabla \risk(\hstar) = 0$, we get
\begin{align}
    \mean_n[\norm{h_{n+1} - \hstar}_C^2] &\leq \norm{h_n - \hstar}_C^2 - \frac{2\alpha_n}{\Lip(\nabla\risk)} \norm{g_n}^2 + 2 \alpha_n^2 (1 + 2c) \norm{g_n}^2. \nonumber\\ 
    &= \norm{h_n - \hstar}_C^2 + 2\alpha_n \norm{g_n}^2 \left(
        \alpha_n (1 + 2c) - \frac{1}{\Lip(\nabla\risk)}
    \right).
\end{align}
If $\alpha_n \leq \frac{1}{\Lip(\nabla\risk)(1 + 2c)}$, then $\alpha_n (1 + 2c) - \frac{1}{\Lip(\nabla \risk)} \leq 0$, implying $\mean_n[\norm{h_{n+1} - \hstar}_C^2] \leq \norm{h_n - \hstar}_C^2$.
Proceeding by induction, we conclude that 
\begin{equation*}
\mean[\norm{h_n - \hstar}_C^2] \leq \norm{h_1 - \hstar}_C^2
\quad
\text{for all}
\quad
1 \leq n \leq N.
\end{equation*}
Since for $\lambda_k = t_k$ we have $\norm{\cdot} \leq \norm{\cdot}_C$, this implies
\begin{align*}
    \mean[\norm{g_n}^2] &\leq \Lip(\nabla\risk)^2 \mean[ \norm{h_n - \hstar}^2] \\
    &\leq \Lip(\nabla\risk)^2 \mean[\norm{h_n - \hstar}_C^2] \leq \Lip(\nabla\risk)^2 \norm{h_1 - \hstar}_C^2,
\end{align*}
for all $1 \leq n \leq N$.
Now simply apply this inequality to \eqref{eq:risk bound} and conclude.

\subsection{Proof of \Cref{prop:t_i for hstar}}
\label{sec:proof of t_i for hstar}

Let $\theta_i = \dotprod{\hstar, e_i}^2$,
and set $A_i = \sum_{j\geq i} \theta_j$.
Since $A_i \to 0$ as $i \to +\infty$, there exists a strictly increasing sequence $\{n_k\}_{k\in\bbN}$ such that $n_k \geq 2$ and $A_{n_k} \leq 2^{-k}$ for all $k \in \bbN$.
Set $n_0 = 1$ and define $\{t_i\}_{i\in\bbN}$ by setting $t_i = 1/k$ for $ n_{k-1} \leq i < n_{k}$.
It is immediate that $t_1 = 1$, $t_i$ is non-increasing and $t_i \to 0$ as $i \to +\infty$. Then
\begin{equation*}
    \sum_{i=1}^\infty \frac{\dotprod{\hstar, e_i}^2}{t_i}
    = \sum_{i=1}^\infty \frac{\theta_i}{t_i}
    = \sum_{k=1}^\infty \sum_{i=n_{k-1}}^{n_k - 1} k \theta_i
    \leq \sum_{k=1}^\infty k A_{n_{k-1}}
    \leq \sum_{k=1}^\infty k 2^{-k + 1} < +\infty.
\end{equation*}
Hence, defining $\Kdist$ by $\Kdist(\{i\}) = t_i - t_{i+1}$, we get $\hstar \in \Chalf(\Hcal)$.

\subsection{Proof of \Cref{prop:linear operator risk good}}
\label{sec:proof of linear operator risk good}
Convexity of $\risk$ is immediate.
To simplify the notation, we now assume that $\Lambda = \partial \Omega$, $f = 0$ and $L$ is linear, from which the general case is a trivial modification.
Recall \eqref{eq:direc deriv linear case}:
\begin{equation}
    \label{eq:direc deriv operator notation}
    D \risk(h ; v) = \dotprod{L[h], L[v]}_{L^2(\Omega)} + \dotprod{\Trace[h], \Trace[v]}_{L^2(\partial \Omega)}.
\end{equation}
Since both $L$ and $\Trace$ are bounded and linear, they have Hilbert adjoints $L^* : L^2(\Omega) \to \bbH^\ell(\Omega)$ and $\Trace^* : L^2(\partial \Omega) \to \bbH^\ell(\Omega)$, allowing us to rewrite \eqref{eq:direc deriv operator notation} as
\begin{align*}
    D \risk (h ; v)
    &= \dotprod{(L^*L + \Trace^*\Trace)[h], v}_{\bbH^\ell(\Omega)}.
\end{align*}    
This equation shows that $\risk$ is Fréchet-differentiable and that $\nabla \risk(h) = (L^*L + \Trace^* \Trace)[h]$, so $\nabla \risk$ is Lipschitz and satisfies $\Lip(\nabla\risk) = \norm{L^*L + \Trace^*\Trace}$.

\subsection{Proof of \Cref{thm:sobolev pre basis}}
\label{sec:proof of sobolev pre basis}
If $\superdomain = \Omega$, then linear independence of $\prebasis$ for any $d \geq 1$ follows directly from the positive-definiteness of $\kernel_{\nu, \eta}$.
If $\superdomain$ is strictly larger than $\Omega$, then one may furthermore use the facts that $\kernel_{\nu, \eta}(\bx_i, \cdot)$ is real-analytic over $\bbR^d\setminus\{\bx_i\}$, that $\bbR^d \setminus \{\bx_1, \ldots, \bx_n\}$ is connected for all $n \in \bbN$ and $d \geq 2$, together with analytic continuation \cite[Theorem~4.8]{stein2003}, to show that $\prebasis$ is still linearly independent.

To show $\prebasis$ has dense span we may assume, without loss of generality, that $\superdomain = \Omega$.
The main result we will use is the following theorem, which is a consequence of \cite[Equation~(4.15)]{rasmussen2008}, \cite[Corollary~10.48]{wendland2004} (see also the discussion before Corollary 11.33) and the extension theorem for fractional Sobolev spaces \cite[Theorem~A.4]{mclean2000}:
\begin{theorem}
    \label{thm:matern RKHS is sobolev}
    Let $\kernel_{\nu, \eta}$ be the Matérn kernel on $\Omega \subset \bbR^d$, with smoothness parameter $\nu$, and let $s = \nu + d/2$.
    Then, the RKHS $\Hcal_{\nu, \eta}$ generated by $\kernel_{\nu, \eta}$ is equivalent to $\bbH^{s}(\Omega)$, meaning that $\Hcal_{\nu, \eta} = \bbH^s(\Omega)$ as a set of functions and their norms are equivalent: there exist constants $c_1$ and $c_2$ such that
    \begin{equation*}
        c_1 \norm{h}_{\bbH^s(\Omega)}
        \leq \norm{h}_{\Hcal_{\nu, \eta}}
        \leq c_2 \norm{h}_{\bbH^s(\Omega)}
        \quad
        \text{for all}
        \quad
        h \in \Hcal_{\nu, \eta}.
    \end{equation*}
\end{theorem}
Now we will prove \Cref{thm:sobolev pre basis}.
We start with some notation: If $(\Ecal, \norm{\cdot}_{\Ecal})$ is a given normed space and $E \subset \Ecal$, we denote by $\closure(E ; \Ecal)$ the closure of $E$ with respect to $\norm{\cdot}_{\Ecal}$.
In this notation, we want to show that, under the assumptions of \Cref{thm:sobolev pre basis}, we have $\closure(\spann \Bcal ; \bbH^\ell(\Omega)) = \bbH^\ell(\Omega)$.
The idea will be to show that this closure actually contains a Sobolev space of order at least $\ell$, which will be dense in $\bbH^\ell(\Omega)$.
We divide the proof into four steps:

First, let $s = \nu + d/2$, which is greater than or equal to $\ell$, by assumption.
We claim that $\closure(\spann \Bcal, \bbH^\ell(\Omega)) \supset \closure(\spann \Bcal, \bbH^s(\Omega))$.
Indeed, since the inclusion $\bbH^s(\Omega) \hookrightarrow \bbH^\ell(\Omega)$ is continuous \cite[Theorem~3.27]{mclean2000}, convergence in $\bbH^s(\Omega)$-norm implies convergence in $\bbH^\ell(\Omega)$-norm.
Next, by \Cref{thm:matern RKHS is sobolev}, we have $\closure(\spann \Bcal, \bbH^s(\Omega)) = \closure(\spann \Bcal, \Hcal_{\nu, \eta})$.
Thirdly, let $\bar{\Bcal} = \{\Phi(\bx)\}_{\bx \in \Omega}$.
We then have $\closure(\spann \Bcal, \Hcal_{\nu, \eta}) = \closure(\spann \bar{\Bcal}, \Hcal_{\nu, \eta}) = \Hcal_{\nu, \eta}$.
Indeed, for $\bx \in \Omega$, let $\{\by_{i}\}_{i \in \bbN} \subset \{\bx_i\}_{i\in\bbN}$ be such that $\by_i \to \bx$ as $i \to +\infty$.
Now compute:
\begin{align*}
    & \norm{\kernel_{\nu,\eta}(\bx, \cdot) - \kernel_{\nu,\eta}(\by_i, \cdot)}_{\Hcal_{\nu, \eta}}^2 \\
    &= \dotprod{\kernel_{\nu,\eta}(\bx, \cdot), \kernel_{\nu,\eta}(\bx, \cdot)} + \dotprod{\kernel_{\nu,\eta}(\by_i, \cdot), \kernel_{\nu,\eta}(\by_i, \cdot)}
    \\ &\qquad - 2 \dotprod{\kernel_{\nu,\eta}(\bx, \cdot), \kernel_{\nu,\eta}(\by_i, \cdot)} \\
    &= \kernel_{\nu, \eta}(\bx, \bx) + \kernel_{\nu, \eta}(\by_i, \by_i) - 2 \kernel_{\nu, \eta}(\bx, \by_i),
\end{align*}
where we have used the reproducing property of $\Hcal_{\nu, \eta}$ in the last step.
Now, note that $\kernel_{\nu, \eta}$ is continuous in $\Omega \times \Omega$, which implies that $\lim \kernel_{\nu, \eta}(\by_i, \by_i) = \lim \kernel_{\nu, \eta} (\bx, \by_i) = \kernel_{\nu, \eta}(\bx, \bx)$, thus $\norm{\kernel_{\nu,\eta}(\bx, \cdot) - \kernel_{\nu,\eta}(\by_i, \cdot)}_{\Hcal_{\nu, \eta}}^2$ goes to zero as $i \to +\infty$.
This proves the first equality.
The second one is the known fact that an RKHS can be constructed as the closure of the $\spann$ of the image of the feature map with respect to the RKHS norm, see \cite[Section~10.2]{wendland2004}.
Finally, we have shown that $\bbH^s(\Omega) = \Hcal_{\nu, \eta} \subset \closure (\spann \Bcal, \bbH^\ell(\Omega))$.
Now, recall that since $\Omega$ is open, bounded and has a Lipschitz boundary, we know that $C^\infty(\bar{\Omega})$ is a dense subset of $\bbH^t(\Omega)$ for every $ t \geq 0$ \cite[Theorem~3.29]{mclean2000}.
Therefore, $\closure(\bbH^s(\Omega), \bbH^\ell(\Omega)) \supset \closure (C^\infty(\bar{\Omega}), \bbH^\ell(\Omega)) = \bbH^\ell(\Omega)$.

\subsection{Proof of \Cref{prop:hjb risk properties}}
\label{sec:proof of hjb risk properties}
Since the functions $(t, \bx) \mapsto T - t$ and $(t, \bx) \mapsto d^{-1} \norm{\bx}^2$ belong to $C^{\infty}(\Omegatilde)$ and have bounded derivatives of all orders, integration by parts shows that $\EnforceBC : \bbH^2(\Omegatilde) \to \bbH^2(\Omegatilde)$ is well-defined, linear and continuous.
Consequently, it is also Fréchet-differentiable.
Therefore, what remains to be shown is that $\Ecal : \bbH^2(\Omegatilde) \to \bbR$ is a well-defined, Fréchet-differentiable, nonconvex functional over $\bbH^2(\Omegatilde)$.
\looseness=-1

Start by noting that, since $\abs{\psi(\bp)} \leq 2 B \norm{\bp}$, we have
\begin{equation}
    \label{eq:htilde bound}
    \abs*{\Htilde(\bx, \bp, \bM)} \leq d^{-1}\norm{\bx}^2 + \norm{\bx}\norm{\bp} + \frac{d}{2} B \norm{\bp} + \frac{1}{2} \abs{\trace (\bM)}.
\end{equation}
for any $(\bx, \bp, \bM) \in \bbR^d \times \bbR^d \times \bbR^{d^2}$.
Hence, since $\Omegatilde$ is bounded, for any $h \in \bbH^2(\Omegatilde)$ the map $(t, \bx) \mapsto \Htilde(\bx, \partial_{\bx} h (t, \bx), \partial_{\bx\bx} h(t, \bx))$ is a member of $L^2(\Omegatilde)$, implying that $\Ecal(h)$ is well-defined for any $h \in \bbH^2(\Omegatilde)$.
Furthermore, it is straightforward to see that $\Htilde$ is concave in $\bp$, implying that $\Ecal$ is not a convex function of $h$.

To show $\Ecal$ is Fréchet differentiable, we first note that the partial derivative of $\Htilde$  with respect to $\bp$ is given by
\begin{equation*}
    \partial_{\bp} \Htilde(\bx, \bp, \bM) =
    \bx - \frac{d}{4} \partial_{\bp} \psi (\bp),
    \quad \text{where} \quad
    \partial_{\bp} \psi(\bp) = \begin{cases}
        2 \bp \quad \text{if} \quad  \norm{\bp} \leq B, \\
        2 B \frac{\bp}{\norm{\bp}} \quad \text{if} \quad \norm{\bp} > B.
    \end{cases}
\end{equation*}
Thus, $\partial_{\bp} \Htilde$ is continuous and bounded by $\sqrt{d} \Xbar + \frac{dB}{2}$ in norm (recall that $\Omegatilde = (0, T) \times (-\Xbar, \Xbar)^d$).
Now we take arbitrary $h, u \in \bbH^2(\Omega)$ and $\delta \in \bbR$ to compute

\begin{align*}
    \Ecal(h + \delta u)
    &= \frac{1}{2} \int_{\Omegatilde} \big[\partial_t [h + \delta u](t, \bx) + \Htilde(\bx, \partial_{\bx}[h + \delta u](t, \bx), \partial_{\bx\bx}[h + \delta u](t, \bx)) \big]^2 \drm t \drm \bx \\
    &\hspace{1cm}
    + \frac{1}{2} \int_{[-\Xbar, \Xbar]^d} \big[ [h + \delta u](t, \bx) - d^{-1}\norm{\bx}^2 \big]^2 \drm \bx \\
    &= \frac{1}{2} \int_{\Omegatilde} \Big[ \partial_t [h + \delta u](t, \bx) + \Htilde(\bx, \partial_{\bx} h (t, \bx), \partial_{\bx\bx} h (t, \bx)) \\
    &\hspace{1.5cm}
    + \delta \gap \partial_{\bp} \Htilde(\bx, \bp_\delta(t, \bx), \partial_{\bx\bx} h(t, \bx) )^\trp \gap \partial_{\bx} u (t, \bx) + \frac{\delta}{2} \trace(\partial_{\bx\bx} u(t, \bx)) \Big]^2 \drm t \drm \bx \\
    &\hspace{1cm}
    + \frac{1}{2} \int_{[-\Xbar, \Xbar]^d} \big[ [h + \delta u](T, \bx) - d^{-1}\norm{\bx}^2 \big]^2 \drm \bx,
\end{align*}

where $\bp_\delta(t, \bx)$ lies on the line segment connecting $\partial_{\bx} h (t, \bx)$ and $\partial_{\bx} [h + \delta u](t, \bx)$, and is due to the Mean Value Theorem.
Now we expand:
\begin{align*}
    &\Ecal(h + \delta u) = \Ecal(h) + \delta \Bigg\{
    \int_{\Omegatilde} \big[
        \partial_t h(t, \bx) + \Htilde(\bx, \partial_{\bx} h (t, \bx), \partial_{\bx\bx} h (t, \bx))
    \big] \\
    &\hspace{2.0cm} 
    \cdot \bigg[ \partial_t u(t, \bx) + \partial_{\bp} \Htilde (\bx, \bp_\delta(t, \bx), \partial_{\bx\bx} h(t, \bx))^\trp \partial_{\bx} u (t, \bx) + \frac{1}{2} \trace(\partial_{\bx\bx} u (t, \bx)) \bigg] \drm t \drm \bx \\
    &\hspace{2.5cm}
    + \int_{[-\Xbar, \Xbar]^d} \big[ h(T, \bx) - d^{-1}\norm{\bx}^2 \big] \cdot u(T, \bx) \drm \bx
    \Bigg\} \\
    &\hspace{1.2cm}
    + \frac{\delta^2}{2} \Bigg\{
        \int_{\Omegatilde} \Big[ \partial_t u(t, \bx) +  \partial_{\bp} \Htilde (\bx, \bp_\delta(t, \bx), \partial_{\bx\bx} h(t, \bx))^\trp \partial_{\bx} u(t, \bx) \\
        &\hspace{4.5cm}
        + \frac{1}{2} \trace(\partial_{\bx\bx}u(t, \bx)) \Big]^2 \drm t \drm \bx + \int_{[-\Xbar, \Xbar]^d} u(T, \bx)^2 \drm \bx
    \Bigg\},
\end{align*}
obtaining:
\begin{align*}
    &\frac{\Ecal(h + \delta u) - \Ecal(h)}{\delta}
    = \int_{\Omegatilde} \big[
        \partial_t h(t, \bx) + \Htilde(\bx, \partial_{\bx} h (t, \bx), \partial_{\bx\bx} h(t, \bx))
    \big] \\
    &\hspace{1.5cm}
    \cdot \bigg[ \partial_t u(t, \bx) + \partial_{\bp} \Htilde (\bx, \bp_\delta(t, \bx), \partial_{\bx\bx}h(t, \bx))^\trp \partial_{\bx} u (t, \bx)
    + \frac{1}{2} \trace(\partial_{\bx\bx}u(t, \bx)) \bigg] \drm t \drm \bx \\
    &\hspace{1cm}+ \int_{[-\Xbar, \Xbar]^d} \big[ h(T, \bx) - d^{-1}\norm{\bx}^2 \big] \cdot u(T, \bx) \drm \bx \\
    &\hspace{1cm}
    + \frac{\delta}{2} \Bigg\{
        \int_{\Omegatilde} \big[ \partial_t u (t, \bx) + \partial_{\bp} \Htilde (\bx, \bp_\delta(t, \bx), \partial_{\bx\bx} h(t, \bx))^\trp \partial_{\bx} u(t, \bx) \\
        &\hspace{2cm}
        + \frac{1}{2} \trace(\partial_{\bx\bx}u(t, \bx)) \big]^2 \drm t \drm \bx + \int_{[-\Xbar, \Xbar]^d} u(T, \bx)^2 \drm \bx
    \Bigg\}.
\end{align*}
Thus, since $\partial_{\bp}\Htilde$ is continuous and bounded in norm by a constant, we can apply dominated convergence as we let $\delta \to 0$ to get
\begin{align*}
    D \Ecal (h ; u)
    &= \lim_{\delta \to 0}
    \frac{\Ecal(h + \delta u) - \Ecal(h)}{\delta} \\
    &= \int_{\Omegatilde} \big[
        \partial_t h(t, \bx) + \Htilde(\bx, \partial_{\bx} h (t, \bx), \partial_{\bx\bx} h(t, \bx))
    \big] \\*
    &\hspace{0.5cm}
    \cdot \bigg[ \partial_t u(t, \bx) + \partial_{\bp} \Htilde (\bx, \partial_{\bx} h (t, \bx), \partial_{\bx\bx}h(t, \bx))^\trp \partial_{\bx} u (t, \bx)
    + \frac{1}{2} \trace(\partial_{\bx\bx}u(t, \bx)) \bigg] \drm t \drm \bx \\*
    &\hspace{0.2cm}+ \int_{[-\Xbar, \Xbar]^d} \big[ h(T, \bx) - d^{-1}\norm{\bx}^2 \big] \cdot u(T, \bx) \drm \bx.
\end{align*}
Again, since $\partial_{\bp} \Htilde$ is bounded and $\Htilde$ satisfies \eqref{eq:htilde bound}, we conclude that $D \Ecal(h ; \cdot)$ is a bounded linear functional on $\bbH^2(\Omegatilde)$ for any $h \in \bbH^2(\Omegatilde)$.
Furthermore, by using Cauchy-Schwarz as well as the fact that both $\Htilde$ and $\partial_{\bp} \Htilde$ are Lipschitz continuous in $\bp$ and $\bM$, one may show that the map $h \mapsto D(h ; \cdot)$, from $\bbH^2(\Omegatilde)$ to its dual, is also continuous.
This implies the Fréchet differentiability of $\Ecal$ \cite{pathak2018}.

\end{document}